\documentclass[11pt, reqno]{amsart} \textwidth=14.9cm \oddsidemargin=1cm \evensidemargin=1cm
\usepackage{amsmath,amstext,amsbsy,amssymb,amscd}
\usepackage{amsmath}
\usepackage{amsxtra}
\usepackage{amscd}
\usepackage{amsthm}
\usepackage{amsfonts}
\usepackage{amssymb}
\usepackage{eucal}
\usepackage{color}
\usepackage{graphicx}%
\newtheorem{thm}{Theorem}[section]
\theoremstyle{plain}

\newtheorem{lem}[thm]{Lemma}
\newtheorem{prop}[thm]{Proposition}
\newtheorem{cor}[thm]{Corollary}

\theoremstyle{definition}
\newtheorem{defn}[thm]{Definition}
\newtheorem{example}[thm]{Example}

\theoremstyle{remark}
\newtheorem{rem}[thm]{Remark}

\definecolor{A}{rgb}{.75,1,.75}

\numberwithin{equation}{section}

\newcommand{\mZ}{\mathcal{Z}}
\newcommand{\be}{\beta}

\newcommand{\ga}{\gamma}
\newcommand{\Ga}{\Gamma}
\newcommand{\va}{\varepsilon}
\newcommand{\La}{\Lambda}
\newcommand{\ovG}{\overline{G}}
\newcommand{\F}{\mathbb{F}}

{\vskip-\lastskip\medskip
  \noindent
  {\em #1.}\enspace
  }%
{\qed\par\medskip
  }

\begin{document}

\title[Wreath Hecke algebras and centralizer construction]{Wreath Hecke
algebras and centralizer construction for wreath products}

\author{Jinkui Wan}
\address{Department of Mathematics, University of Virginia,
Charlottesville,VA 22904, USA.} \email{jw5ar@virginia.edu}

\begin{abstract}
Generalizing the centralizer construction of Molev and Olshanski
on symmetric groups, we study the structures of the centralizer
$\mZ_{m,n}$ of the wreath product $G_{n-m}$ in the group algebra
of $G_n$ for any $n\geq m$. We establish the connection between
$\mZ_{m,n}$ and a generalization of degenerate affine Hecke
algebras introduced in our earlier work.
\end{abstract}

\maketitle \setcounter{tocdepth}{1} \tableofcontents

\section{Introduction}
\subsection{} Let $\F$ be an arbitrary field. It is well known that
there is a canonical surjective homomorphism from the degenerate
affine Hecke algebra $\mathcal{H}_m$ introduced by
Drinfeld~\cite{Dr} and Lusztig~\cite{Lus} to the group algebra $\F
S_m$ of the symmetric group $S_m$, where the polynomial generators
are mapped to the Jucys-Murphy elements in $\F S_m$. This has
played a key role in a new approach to the representation theory
of symmetric groups in arbitrary characteristic ~(cf. Kleshchev's
book~\cite{K}). Let $G$ be a finite group and $G_m$ be the wreath
product of $G$ with the symmetric group $S_m$. In a joint work
with Wang~\cite{WW}, the author has further explored modular
representations of wreath products $G_m$ by introducing wreath
Hecke algebras $\mathcal{H}_m(G)$. When $G=C_r$ a cyclic group,
$\mathcal{H}_m(C_r)$ also appeared in Ram and Shepler \cite{RS} in
their search of degenerate (=graded) Hecke algebras associated to
complex reflection groups. As in the case of symmetric groups,
there exists a surjective homomorphism from $\mathcal{H}_m(G)$ to
$\F G_m$ with the polynomial generators being mapped to the
generalized Jucys-Murphy elements in $\F G_m$ (introduced
independently in~\cite{Pu} and~\cite{W1} with different
applications).

On the other hand, Olshanski~\cite{O1} shows that there is a
natural homomorphism from the degenerate affine Hecke algebra
$\mathcal{H}_m$ to the centralizer of $S_{n-m}$ in the group
algebra $\mathbb{F}S_n$ for any $n\geq m$. This fact is used by
Okounkov and Vershik~\cite{OV} to develop a new approach to the
representation theory of the symmetric groups over the complex
field $\mathbb{C}$. In~\cite{MO}, Molev and Olshanski further
studied the centralizer of $S_{n-m}$ in $\mathbb{F}S_n$ and
establish its deeper connection with $\mathcal{H}_m$.

The main goal of the present work is to study the algebraic
structure of the centralizer $\mZ_{m,n}$ of the wreath product
$G_{n-m}$ in $\F G_n$ and investigate its interrelation with the
wreath Hecke algebra $\mathcal{H}_m(G)$. We are much inspired by
the approach of centralizer constructions for Lie algebras
$\mathfrak{gl}(n)$ in~\cite{O1} and symmetric groups $S_n$
in~\cite{MO}. Our work specializes to~\cite{MO} when $G=\{1\}$,
however we need to generalize various concepts from symmetric
groups and develop new delicate combinatorial analysis in order to
treat the extra complications coming from the presence of the
group $G$.



\subsection{} The approach of centralizer construction is proposed
by Olshanski in~\cite{O1}, where the centralizer
$\mathcal{A}_m(n)$ of the subalgebra $\mathfrak{gl}(n-m)$ in the
enveloping algebra $U(\mathfrak{gl}(n))$ is studied. There exists
a chain of algebra homomorphisms
\begin{equation}\label{chain1}
\cdots\rightarrow\mathcal{A}_m(n)\rightarrow\mathcal{A}_m(n-1)
\rightarrow\cdots\rightarrow\mathcal{A}_m(m+1)\rightarrow\mathcal{A}_m(m),
\end{equation}
associated to which one can introduce and study the projective
limit algebra $\mathcal{A}_m$. It turns out that the Yangian
$Y(\mathfrak{gl}(m))$ \cite{O2} for $\mathfrak{gl}(m)$ appears as a
subalgebra of $\mathcal{A}_m$ and there is a surjective
homomorphism from $Y(\mathfrak{gl}(m))$ to the centralizer
$\mathcal{A}_m(n)$.
%
%
%
%
When Molev and Olshanski tried to extend the construction to
symmetric groups, they encountered the new difficulty that there
does not exist a natural analog of the chain~(\ref{chain1}) for
the centralizers of $S_{n-m}$ in $\F S_n$. However it turns out to
work well if one takes the semigroup algebra $\F \Ga(n)$ instead
of the group algebra $\F S_n$, where $\Ga(n)$ is the semigroup of
$(0,1)$-matrices which have at most one $1$ in each row and
column~(cf. \cite{MO}).

When we extend the work to wreath products, the same problem
arises and hence similarly we shall begin with the semigroup
$\ovG_n$ consisting of $n\times n$-matrices with entries in
$G\cup\{0\}$ such that any row and column have at most one nonzero
entry, which specializes to the semigroup $\Ga(n)$ when $G=\{1\}$.
Indeed we have the canonical projection
$\theta_n:\ovG_{n}\rightarrow\ovG_{n-1}$ which maps
$\gamma\in\ovG_n$ to the upper left corner of $\gamma$ of order
$n-1$. Fix a nonnegative integer $m$ and for any $n\geq m$, denote
by $\overline{\mZ}_{m,n}$ the centralizer of $\ovG_{n-m}$ in the
semigroup algebra $\mathbb{F}\ovG_n$. The projection
$\theta_n:\ovG_{n}\rightarrow\ovG_{n-1}$ gives rise to a chain of
homomorphisms for the algebras $\overline{\mZ}_{m,n}$ as:
$$
{}\cdots\rightarrow \overline{\mZ}_{m,n}\rightarrow
\overline{\mZ}_{m,-1}\rightarrow\cdots\rightarrow
\overline{\mZ}_{m,m+1}\rightarrow \overline{\mZ}_{m,m}=\F\ovG_m.
$$
We define the algebra $\overline{\mZ}_m$ as the corresponding
projective limit.
We show that the algebra $\overline{\mZ}_m$ has a decomposition as
:
$$
\overline{\mZ}_m\cong\overline{\mZ}_0\otimes\overline{\mathcal{H}}_m(G),
$$
where $\overline{\mathcal{H}}_m(G)$ is a "semigroup analog" of the
wreath Hecke algebra $\mathcal{H}_m(G)$ in~\cite{WW}. We obtain an
algebra homomorphism
$\overline{\Psi}:\overline{\mathcal{H}}_m(G)\rightarrow
\overline{\mZ}_{m,n}$ for any $n\geq m$ and formulate a
presentation of $\overline{\mathcal{H}}_m(G)$ by generators and
relations. We show that there exists a natural surjective algebra
homomorphism from $\F\ovG_n$ to $\F G_n$ and moreover it induces
epimorphisms
$\overline{\Phi}:\overline{\mathcal{H}}_m(G)\rightarrow\mathcal{H}_m(G)$
and $\Phi:\overline{\mZ}_{m,n}\rightarrow \mZ_{m,n}$. Hence we
obtain a natural homomorphism $\Psi:\mathcal{H}_m(G)\rightarrow
\mZ_{m,n}$ and the following commutative diagram:
$$\unitlength=1cm
\begin{picture}(6,2.5)
\put(1.1,2){$\overline{\mathcal{H}}_m(G)$}
\put(3.9,2){$\overline{\mZ}_{m,n}$}
\put(1.1,0.2){$\mathcal{H}_m(G)$} \put(3.9,0.2){$\mZ_{m,n}$}
\put(2.3,2.1){\vector(1,0){1.5}} \put(2.3,0.3){\vector(1,0){1.5}}
\put(2.65,2.2){$\overline{\Psi}$} \put(2.65,0.4){$\Psi$}
\put(1.5,1.85){\vector(0,-1){1.3}}
\put(4.4,1.85){\vector(0,-1){1.3}}
\put(1.03,1.1){$\overline{\Phi}$} \put(4.6,1.1){$\Phi$}
\end{picture}$$

As an application, we consider the case when $\F=\mathbb{C}$ and
show that the Gelfand-Zetlin subalgebra of $\mathbb{C}G_n$ is
generated by the Jucys-Murphy elements in $\mathbb{C}G_n$ and a
maximal commutative subalgebra of $\mathbb{C}G^n$. It is
semisimple and a maximal commutative subalgebra of the group
algebra $\mathbb{C}G_n$. This is a wreath analog of the result
in~\cite[Theorem~11]{O1} for symmetric groups.


\subsection{} The paper is organized as follows. In
Section~\ref{Cc}, we introduce the semigroup $\ovG_n$ which is a
generalization of the wreath product $G_n$. For each $m\geq 0$ we
construct the algebras $\overline{\mZ}_m$ as projective limits of
the centralizers $\overline{\mZ}_{m,n}$ of $\ovG_{n-m}$ in
$\F\ovG_n$ for all $n\geq m$. In Section~\ref{mZ0}, we give a
linear basis for the algebra $\overline{\mZ}_0$. In
Section~\ref{mZm}, we investigate the algebraic structure of
$\overline{\mZ}_m$ for $m>0$. In Section~\ref{CwRh}, we give an
interrelation between $\mZ_{m,n}$ and  wreath Hecke algebras
$\mathcal{H}_m(G)$.

{\bf Acknowledgements.} I am deeply grateful to my advisor W. Wang
for many helpful suggestions for this paper.


\section {The centralizer construction}\label{Cc}

\subsection{The wreath products}
Let $G$ be a finite group with unity $1$, and $G_*=\{C_1,\ldots,
C_r\}$ be the set of all conjugacy classes of $G$ with
$C_1=\{1\}$. The symmetric group $S_n$ acts on the product group
$G^n=G\times \cdots\times G$ by permutations:
$$
{}^{w}g:=w (g_1,\ldots, g_n)=(g_{w^{-1}(1)},\ldots, g_{w^{-1}(n)})
$$
for  $g=(g_1,\ldots,g_n)\in G^n$ and $w\in S_n$. The wreath
product of $G$ with $S_n$ is defined to be the semidirect product
$$
G_n=G^n\rtimes S_n=\{(g,w)|~g=(g_1,\ldots, g_n)\in G^n, w\in S_n\}
$$
with the multiplication
$$
(g,w)(h,\tau)=({}^{w}g\,h, w\tau).
$$

Let $\lambda=(\lambda_1,\ldots, \lambda_l)$ be a partition of
integer $|\lambda|=\lambda_1+\cdots+\lambda_l$, where
$\lambda_1\geq \dots \geq \lambda_l \geq 1$.
We will identify the partition $(\lambda_1, \lambda_2, \ldots,
\lambda_l)$ with $(\lambda_1, \lambda_2, \ldots, \lambda_l, 0,
\ldots, 0)$. We will also write a partition as $
\lambda=(1^{m_1}2^{m_2}\cdots), $ where $m_i$ is the number of
parts in $\lambda$ equal to $i$.

We will use partitions indexed by $G_*$. For a finite set $X$ and
$\rho=(\rho(x))_{x\in X}$ a family of partitions indexed by $X$,
we write
$$\|\rho\|=\sum_{x\in X}|\rho(x)|.$$
Sometimes it is convenient to regard $\rho=(\rho(x))_{x\in X}$ as
a partition-valued function on $X$. We denote by $\mathcal P(X)$
the set of all partitions indexed by $X$ and by $\mathcal P_n(X)$
the set of all partitions in $\mathcal P(X)$ such that
$\|\rho\|=n$.

The conjugacy classes of $G_n$ can be described as follows. Let
$x=(g, \sigma )\in G_n$, where $g=(g_1, \cdots, g_n) \in G^n,$ $
\sigma \in S_n$. The permutation $\sigma $ is written as a product
of disjoint cycles. For each such cycle $y=(i_1 i_2 \cdots i_k)$
the element $g_{i_k} g_{i_{k -1}} \cdots g_{i_1} \in G$ is
determined up to conjugacy in $G$ by $g$ and $y$, and will be
called the {\em cycle-product} of $x$ corresponding to the cycle
$y$. For any conjugacy class $C$ and each integer $i\geq 1$, the
number of $i$-cycles in $\sigma$ such that the corresponding
cycle-product lies in $C$ will be denoted by $m_i(C)$. Denote by
$\rho (C)$ the partition $(1^{m_1 (C)} 2^{m_2 (C)} \ldots )$, $C
\in G_*$. Then each element $x=(g, \sigma)\in G_n$ gives rise to a
partition-valued function $( \rho (C))_{C \in G_*} \in {\mathcal
P} ( G_*)$ such that $\sum_{i, C} i m_i(C) =n$. The
partition-valued function $\rho =( \rho(C))_{ C \in G_*}
=(\rho(C_1),\ldots,\rho(C_r))$ is called the {\em type} of $x$. It
is known (cf. \cite[Section 4.2]{JK}, \cite[Chapter I, Appendix B]{Mac})
that any two elements of $G_n$ are
conjugate in $G_n$ if and only if they have the same type.
\subsection{Jucys-Murphy elements}
Note that for each $1\leq m\leq n$, $G_m$ embeds in $G_n$
canonically as the subgroup $G_m\times 1$. On the other hand,
$G_m$ can also embed in $G_n$ as the subgroup $1\times G_m$, which
we shall denote by $G^{\prime}_m$; namely
\begin{equation}\label{eqn:G_m'}
G^{\prime}_m:=\{((g_1,\cdots,g_n),\sigma)|~g_i=1 \text{ for }1\leq
i\leq n-m, \sigma \text{ fixes }1,\ldots,n-m\}.
 \end{equation}
 For each $h\in G$ and $1\leq i\leq n$, let
$h^{(i)}\in G^n$ correspond to $h$ in the $i$th factor group
$G^{(i)}$ of $G^n$. For each $1\leq k<l\leq n$, set
$$
t_{kl}=\sum_{h\in G}h^{(k)}(h^{-1})^{(l)}.
$$
Recall the (generalized) {\it Jucys-Murphy elements} $\xi_k\in
\mathbb{F}G_n (1\leq k\leq n)$ are introduced independently in
\cite{Pu} and \cite{W1} as follows:
\begin{eqnarray}\label{eqn:jucys}
\xi_k: =\sum_{l=k+1}^{n}\sum_{h\in G} \left(h^{(k)}(h^{-1})^{(l)},
(k,l) \right)=\sum_{ l=k+1}^{n}t_{kl}(k,l),
\end{eqnarray}
where $(k,l)$ is the transposition permuting $k$ and $l$ in the
symmetric group $S_n$. Note that $\xi_n=0$ and $\xi_k$ commutes
with $G^{\prime}_{n-k}$. As
$\xi_k\in\mathbb{F}G^{\prime}_{n-k+1}$, it follows that the
Jucys-Murphy elements commute.
\subsection{The semigroups $\ovG_n$ }
We shall often identify the symmetric group $S_n$ with the group
consisting of $(0,1)$-matrices of size $n\times n$ which have
exactly one $1$ in each row and each column, where the product is
the matrix multiplication. Indeed each permutation $\sigma\in S_n$
corresponds to the $n\times n$ matrix $M(\sigma)$ as follows:
\begin{eqnarray*}
M(\sigma)_{ij}=\left\{
 \begin{array}{ll}
 1, & \text{ if }\sigma(j)=i,\\
 0, &\text{ otherwise}.
 \end{array}
 \right.
\end{eqnarray*}

Now we consider the union $G\cup\{0\}$, where $0$ is an extra
symbol, and adopt the convention that
\begin{equation}
g0=0g=0,g+0=0+g=g\text{ for any }g\in G.\label{multcon}
\end{equation}
Then the wreath product $G_n$ can be identified with the group
consisting of $n\times n$ matrices with entries in $G\cup\{0\}$
such that any row and column have {\bf exactly} one nonzero entry,
where the product is the matrix multiplication with the
convention~(\ref{multcon}). In fact, the map sending $g\in G^n$
and $\sigma\in S_n$ to $\Delta(g)$ and $M(\sigma)$ respectively is
a group isomorphism, where $\Delta(g)$ is the diagonal matrix with
$\Delta(g)_{ii}=g_i$ for any $g=(g_1,\ldots,g_n)\in G^n$ and
$1\leq i\leq n$.

\begin{defn}
For each $n\in\{1,2,\ldots\}$, define $\ovG_n$ to be the semigroup
which consists of $n\times n$ matrices with entries in
$G\cup\{0\}$ such that any row and column have {\bf at most} one
nonzero entry, where the product is the matrix multiplication with
the convention~(\ref{multcon}). We set $\ovG_0=\{1\}$. Note that
the wreath product $G_n$ is a subgroup of the semigroup $\ovG_n$.
\end{defn}
\begin{rem} The semigroup $\ovG_n$ is the specialization of the
semigroup $\Ga(n,D)$ introduced in~\cite[Definition 5.3]{MO} to
the case when $D=G$. Moreover if $G=\{1\}$, then $\ovG_n$
coincides with the semigroup $\Ga(n)$ of partial bijections on the
set $\{1,\ldots,n\}$ introduced in~\cite[Definition 2.6]{MO}.
\end{rem}
\begin{defn}\label{defn:proj.inj.}
For $1\leq m\leq n$, we define two mappings as follows.
\begin{enumerate}
\item The {\it canonical projection } $\theta: \ovG_n\rightarrow
\ovG_m$. For $\ga\in \ovG_n$,  $\theta(\ga)$ is the upper left
corner of $\ga$ of order $m$.

\item The {\it canonical embedding} $\phi:\ovG_m\hookrightarrow
\ovG_n$. For each $\ga\in \ovG_m$, $\phi(\ga)$ is defined as
follows:
\begin{eqnarray*}
\phi(\ga)_{ij}=\left\{
 \begin{array}{ll}
 \ga_{ij}, & \text{ if }1\leq i,j\leq m,\\
 \delta_{ij}, &\text{ otherwise}.
 \end{array}
 \right.
\end{eqnarray*}
%
%
%

%
%
\end{enumerate}
\end{defn}
%

%

For $1\leq i\leq n$, let us denote by $\va_i$ the diagonal
$n\times n$-matrix whose $(i,i)$th entry is $0$ and all other
diagonal entries are equal to 1. Note that the semigroup $\ovG_n$
is generated by $G_n$ and the elements $\va_i$, $1\leq i\leq n$.
Moreover we have for any $1\leq i\leq n$ and $h\in G, g\in G^n,
\sigma\in S_n$,
\begin{equation}\label{eqn:semigprel}
h^{(i)}\va_i=\va_ih^{(i)}=\va_i,\quad g\va_i=\va_ig,\quad
\sigma\va_i=\va_{\sigma(i)}\sigma.
\end{equation}
%
Let $s_i=(i,i+1)\in S_n$ be the transposition for each $1\leq
i\leq n-1$. It is known that $S_n$ is generated by
$s_1,\ldots,s_{n-1}$ subject to the relation
\begin{equation}\label{reln:braid}
s_i^2=1,\quad s_is_j=s_js_i, \quad
s_is_{i+1}s_i=s_{i+1}s_is_{i+1}, \quad|i-j|>1.
\end{equation}
\begin{prop}\label{prop:pre. of ovGn} The semigroup algebra $\F\ovG_n$ of $\ovG_n$
is isomorphic to the abstract algebra generated by $
s_1,\ldots,s_{n-1},\va_1,\ldots,\va_n$ and all $g\in G^n$ subject
to~(\ref{reln:braid}) and the following relations:
\begin{align}
s_ig&={}^{s_i}g\,s_i,\label{reln:symm.gp} \\
\va_i^2=\va_i,\quad \va_i\va_j&=\va_j\va_i,\label{reln:va}\\
s_i\va_i=\va_{i+1}s_i,\quad s_i\va_j&=\va_js_i,
\quad s_i\va_i\va_{i+1}=\va_i\va_{i+1},\quad j\neq i,i+1,\label{reln:symm.va.}\\
\quad \va_ig=g\va_i, \quad \va_i h^{(i)}&=h^{(i)}\va_i=\va_i,
\quad \forall g\in G^n,  h\in G. \label{reln:gp.va}
\end{align}
\end{prop}
\begin{proof} Let us denote by $\mathcal{A}(n,G)$ the abstract
algebra in the proposition.  Note
that~(\ref{reln:braid})-(\ref{reln:gp.va}) hold in $\F\ovG_n$. So
there exists a surjective algebra homomorphism from
$\mathcal{A}(n,G)$ to $\F\ovG_n$. Then it suffices to show the
dimension of $\mathcal{A}(n,G)$ is no bigger than that of
$\F\ovG_n$. On the one hand, we have
\begin{eqnarray}\label{eqn:dim.FovGn}
\text{dim}_{\F} \F\ovG_n=|\ovG_n|=\sum_{t=0}^{n}\left(
\begin{array}{l}
n \\
t
\end{array}
\right)^2 t!|G|^t.
\end{eqnarray}
On the other hand, it follows from the
relations~(\ref{reln:symm.va.})~and~(\ref{reln:gp.va}) that
$$
\mathcal{A}(n,G)=\mathbb{F}G_n\cdot\mathbb{F}\langle\va_1,\ldots,\va_n\rangle,
$$
where $\mathbb{F}\langle\va_1,\ldots,\va_n\rangle$ is the
subalgebra of $\mathcal{A}(n,G)$ generated by
$\va_1,\ldots,\va_n$. Observe that relation~(\ref{reln:va})
implies that $\mathbb{F}\langle\va_1,\ldots,\va_n\rangle$ is
spanned by the monomials $\va_{i_1}\cdots\va_{i_t}$ with
$i_1<\cdots<i_t$ and $0\leq t\leq n$. Next we shall estimate the
dimension of the subspace $\mathbb{F}G_n\va_{i_1}\cdots\va_{i_t}$.
Using~(\ref{eqn:semigprel}) we may assume  $i_l=l$ for each $1\leq
l\leq t$ without loss of generality. By
~(\ref{reln:symm.va.})~and~(\ref{reln:gp.va}) we have in
$\mathcal{A}(n,G)$
$$
\be G_t\va_1\cdots\va_t=\be \va_1\cdots\va_t, \quad \forall \be\in
G_n.
$$
Here $G_t$ is regarded as a subgroup of $G_n$. Hence the dimension
of the subspace $\mathbb{F}G_n\va_1\cdots\va_t$ is no bigger than
the number of left cosets of $G_t$ in $G_n$. Therefore,
$\text{dim}_{\mathbb{F}}\mathcal{A}(n,G)$ does not exceed
$$
\sum_{t=0}^n\left(
\begin{array}{l}
n \\
t
\end{array}
\right)\frac{n!}{t!}|G|^{n-t}
$$
which coincides with~(\ref{eqn:dim.FovGn}).
\end{proof}

\begin{cor}\label{map:retraction} The mapping
\begin{align}
\Phi: \F\ovG_n&\rightarrow \mathbb{F}G_n,\\
 g\mapsto g, s_i\mapsto
(i,i+1), \va_j&\mapsto 0,\quad \forall g\in G^n, 1\leq i\leq n-1,
1\leq j\leq n\notag
\end{align}
is an algebra homomorphism, which is an extension of the identity
map of the subalgebra $\mathbb{F}G_n$.
\end{cor}
%

\subsection{The construction of the limit centralizer algebras $\overline{\mZ}_m$.}
Note that the canonical embedding $\phi$ in the case $m=n-1$ in
Definition~\ref{defn:proj.inj.} can be extended to algebra
embeddings $\F\ovG_{n-1}\hookrightarrow \F\ovG_n$ by linearity. In
the remainder of the paper, the embedding induced by $\phi$ will
be referred to when an element in $\F\ovG_{n-1} (\text{resp. }
\ovG_{n-1})$ is regarded as belonging to $\F\ovG_n (\text{resp. }
\ovG_n)$.

We shall denote by $\theta_n$ the canonical projection
$\ovG_n\rightarrow \overline{G}_{n-1}$ in
Definition~\ref{defn:proj.inj.}(1). We extend $\theta_n$ to a
linear map $\F\ovG_n\rightarrow \mathbb{F}\overline{G}_{n-1}$,
still denoted by $\theta_n$, which is not an algebra homomorphism.

For any $0\leq m\leq n$, denote by $\ovG_{n-m}^{\prime}$  the
subsemigroup of $\ovG_n$ which consists of the matrices with first
$m$ diagonal entries equal to 1.
Set the centralizer of $\ovG_{n-m}^{\prime}$ in $\F\ovG_n$ to be
$$
\overline{\mZ}_{m,n}=\F\ovG_n^{\ovG_{n-m}^{\prime}}.
$$
%
Similarly, let us set the centralizer of $G_{n-m}^{\prime}$ in $\F
G_n$ to be
$$
\mZ_{m,n}=\F G_n^{G_{n-m}^{\prime}}.
$$
In particular, $\overline{\mZ}_{0,n}$ and $\mZ_{0,n}$ are the
centers of $\F\ovG_n$ and of $\F G_n$, respectively.
Observe that the algebra homomorphism $\Phi:\F\ovG_n\rightarrow\F
G_n$ in Corollary~\ref{map:retraction} maps the subalgebra
$\F\ovG_{n-m}^{\prime}$ to $\F G_{n-m}^{\prime}$ and hence we have the
following lemma.
\begin{lem}\label{lem:ovmZmn to Zmn} The algebra homomorphism
$\Phi:\F\ovG_n\rightarrow\F G_n$ in Corollary~\ref{map:retraction}
maps $\overline{\mZ}_{m,n}$ onto $\mZ_{m,n}$.
\end{lem}
Let us denote by $\mathbb{Z}_+$ the set of nonnegative integers.
Next we shall introduce the limit centralizer algebras
$\overline{\mZ}_m$ for $m\in\mathbb{Z}_+$.
\begin{lem}\label{lem: hom.theta}
 The restriction of the linear map $\theta_n:\F\ovG_{n}\rightarrow\F\ovG_{n-1}$
 to $\overline{\mZ}_{n-1,n}\subseteq \F\ovG_n$
defines an algebra homomorphism
$\theta_n:\overline{\mZ}_{n-1,n}\rightarrow \F\ovG_{n-1}$.
Moreover,
$$
\theta_n(\overline{\mZ}_{m,n})\subseteq \overline{\mZ}_{m,n-1}.
$$
\end{lem}
\begin{proof}
Observe that  $\theta_n(x)$ coincides with $\va_nx\va_n$ for each
$x\in \F\ovG_n$ if we identify the algebra $\F\ovG_{n-1}$ with the
subalgebra of $\F\ovG_n$ spanned by  $\ga\in\ovG_n$ whose $n$th
row and column are zero. Note that any $x\in
\overline{\mZ}_{n-1,n}=\F\ovG_n^{\ovG_{1}^{\prime}}$ commutes with
$\va_n$ since $\va_n$ is contained in $\ovG_{1}^{\prime}$. Hence,
if $x,y\in \overline{\mZ}_{n-1,n}$, using the fact that $\va_n$ is
an idempotent, then
$$
\theta_n(xy)=\va_nxy\va_n=\va_nx\va_n\va_ny\va_n=\theta_n(x)\theta_n(y).
$$
The second claim can be checked by a direct computation using the
identification of $\F\ovG_{n-1}$ with the subalgebra of $\F\ovG_n$
spanned by $\ga\in\ovG_n$ whose $n$th row and column are zero.
%
\end{proof}
%
%
For each positive integer $n$, set
$\mathbb{N}_n=\{1,2,\ldots,n\}$.

\begin{defn}
 For
$\ga\in\ovG_n$, set
\begin{align}
J(\ga)&=\{k|~1\leq k\leq n, \ga_{kk}\neq 1\},\notag\\
\bar{J}(\ga)&=\{k|~1\leq k\leq n, \ga_{kk}=1\}=\mathbb{N}_n\backslash J(\ga),\label{defn:barJ}\\
\text{deg}(\ga)&=|J(\ga)|,\label{defn: degree}\\
\text{rank}(\ga)&=\{l|~1\leq l\leq n, \ga_{kl}\neq 0 \text{ for
some }1\leq k\leq n\}\label{defn:rank}.
\end{align}
\end{defn}
\begin{rem}
Recall that each element $(g,\sigma)\in G_n$ with
$g=(g_1,\ldots,g_n)\in G^n$ and $\sigma\in S_n$ can be identified
with the matrix $\Delta(g)M(\sigma)$, where $\Delta(g)$ is the
diagonal matrix with $\Delta(g)_{ii}=g_i$ and the matrix
$M(\sigma)$ is as follows:
\begin{eqnarray*}
M(\sigma)_{ij}=\left\{
 \begin{array}{ll}
 1, & \text{ if }\sigma(j)=i,\\
 0, &\text{ otherwise}.
 \end{array}
 \right.
\end{eqnarray*}
Then we have $J((g,\sigma))=\{j\in\mathbb{N}_n|~\sigma(j)\neq j
\text{ or } g_j\neq 1\}$, which is introduced in \cite[Subsection
2.2]{W2}.
\end{rem}
Note that each $\tau\in \F\ovG_n$ can be written as a linear
combination $\tau=\sum_{\ga\in\ovG_n}\tau_{\ga}\ga$ with
$\tau_{\ga}\in\mathbb{F}$, then we set
\begin{equation}\label{defn: degree in AnG}
\text{deg}(\tau)=\text{max}\{\text{deg}(\ga)|~\tau_{\ga}\neq 0\}.
\end{equation}
\begin{lem}\label{lem:degree}
Let $\ga, \ga^{\prime}\in\ovG_n$. Then
$\text{deg}(\ga\ga^{\prime})\leq
\text{deg}(\ga)+\text{deg}(\ga^{\prime})$.
\end{lem}
\begin{proof}
Note that for each $1\leq k\leq n$, if
$\ga_{kk}=1=\ga^{\prime}_{kk}$, then $(\ga\ga^{\prime})_{kk}=1$.
Hence $\bar{J}(\ga)\cap \bar{J}(\ga^{\prime})\subseteq
\bar{J}(\ga\ga^{\prime})$, and therefore
$J(\ga\ga^{\prime})\subseteq J(\ga)\cup J(\ga^{\prime})$.
\end{proof}

For each $0\leq k\leq n$, let $(\F\ovG_n)^k$ be the subspace of
$\F\ovG_n$ spanned by the subset
$\{\ga|~\ga\in\ovG_n,\text{deg}(\ga)\leq k\}$. Then we obtain a
filtration of the space $\F\ovG_n$:
$$
\mathbb{F}=(\F\ovG_n)^0\subseteq
(\F\ovG_n)^1\subseteq\ldots\subseteq (\F\ovG_n)^n=\F\ovG_n.
$$
Lemma~\ref{lem:degree} implies that this filtration is compatible
with the algebra structure of $\F\ovG_n$. Clearly there exists a
filtration on each centralizer $\overline{\mZ}_{m,n}$ inherited
from $\F\ovG_n$. For each $0\leq k\leq n$, denote by
$\overline{\mZ}_{m,n}^k$ the $k$th term of the filtration on
$\overline{\mZ}_{m,n}$. Note that for $\ga\in\ovG_n$ the degree of
$\theta_n(\ga)$ can be equal either to $\text{deg}(\ga)$ or
$\text{deg}(\ga)-1$. Therefore the homomorphism $\theta_n$ in
Lemma~\ref{lem: hom.theta} is compatible with the filtration on
$\F\ovG_n$. This implies
$\theta_n(\overline{\mZ}_{m,n}^k)\subseteq
\overline{\mZ}_{m,n-1}^k$ for $0\leq k\leq n$.
\begin{defn}\label{defn:Am(G)}
For $m=0,1,2,\ldots$, let $\overline{\mZ}_{m}$ be the projective
limit of the sequence
$$
{}\cdots\rightarrow \overline{\mZ}_{m,n}\rightarrow
\overline{\mZ}_{m,n-1}\rightarrow \cdots\rightarrow
\overline{\mZ}_{m,m+1}\rightarrow \overline{\mZ}_{m,m}=\F\ovG_m
$$
taken in the category of filtered algebras. That is, an element
$\alpha\in \overline{\mZ}_{m}$ is a sequence $\{\alpha_n|~n\geq
m\}$ such that
$$
\alpha_n\in \overline{\mZ}_{m,n},\quad
\theta_n(\alpha_n)=\alpha_{n-1}, \quad
\text{deg}(\alpha)=\text{sup}_{n\geq m}\text{deg}(\alpha_n)<\infty
$$
with the componentwise operations.
\end{defn}
Denote by $\theta^n$ the canonical projection
$$\theta^n:\overline{\mZ}_{m}\rightarrow \overline{\mZ}_{m,n},\qquad \alpha=(\alpha_n|n\geq m)\mapsto\alpha_n.$$
The $k$th term of the filtered algebras $\overline{\mZ}_{m}$ will
be denoted by $\overline{\mZ}_{m}^k$.
\begin{lem}\label{lem:ovmZm to FovGm}
For each $m\in\{0,1,2,\ldots\}$, there exists an embedding $
\F\ovG_m \hookrightarrow\overline{\mZ}_{m}$ whose image consists
of constant sequences.
\end{lem}
\begin{proof}
Observe that for each $n\geq m$, $\ovG_m$ commutes with
$\ovG_{n-m}^{\prime}$ in $\ovG_n$ if we identify $\ovG_m$ with its
image in $\ovG_n$ under the canonical embedding induced by $\phi$
in Definition~\ref{defn:proj.inj.}. Hence the mapping $\tau\mapsto
\alpha=(\alpha_n|n\geq m)$ with $\alpha_n=\tau$ defines an algebra
embedding.
\end{proof}

Since $\overline{\mZ}_{m,n}\subseteq \overline{\mZ}_{m+1,n}$, we
have natural injective algebra homomorphisms
$\overline{\mZ}_m\hookrightarrow \overline{\mZ}_{m+1}$ defined by
$ (\alpha_n|n\geq m)\mapsto (\alpha_n|n\geq m+1).$ Let us define
$\overline{\mZ}$ to be the direct limit (the union) of the
algebras $\overline{\mZ}_m$ with respect to the embeddings
$\overline{\mZ}_m\hookrightarrow \overline{\mZ}_{m+1}$.
\begin{prop}
The center of the algebra $\overline{\mZ}$ coincides with
$\overline{\mZ}_0$.
\end{prop}
\begin{proof}
Clearly $\overline{\mZ}_{0,n}$ commutes with
$\overline{\mZ}_{m,n}$ for any $n\geq m$ since
$\overline{\mZ}_{0,n}$ is the center of the algebra $\F\ovG_n$.
Hence $\overline{\mZ}_0$ is contained in the center of
$\overline{\mZ}_m$ as the sequences $(\alpha_n|n\geq
m)\in\overline{\mZ}_m$ are multiplied componentwise. Then it
follows that $\overline{\mZ}_0$ is contained in the center of
$\overline{\mZ}$.  Conversely, suppose $\alpha=(\alpha_n)$ lies in
the center of $\overline{\mZ}$. By Lemma~\ref{lem:ovmZm to FovGm},
there exists a natural embedding $\F\ovG_m\hookrightarrow
\overline{\mZ}$ for any nonnegative integer $m$, and hence
$\alpha$ commutes with any constant sequence coming from
$\F\ovG_m$. This means that each $\alpha_n$ commutes with
$\F\ovG_n$ and hence $\alpha_n\in\overline{\mZ}_{0,n}$ for all
$n$. This implies $\alpha\in\overline{\mZ}_0$.
\end{proof}

\section{The structure of the algebra $\overline{\mZ}_{0}$}\label{mZ0}
In this section we give a linear basis of the commutative algebra
$\overline{\mZ}_0$.

\subsection{Bases of the center $\mZ_{0,n}$ of $\mathbb{F}G_n$.}
For $0\leq k\leq n$, denote by $\mZ_{0,n}^k$ the $k$th term of the
filtration on $\mZ_{0,n}$ inherited from the algebra $\F\ovG_n$.
It can be easily checked that
$$
\mZ_{0,n}^0=\mathbb{F}\cdot1.
$$
 For
any partition-valued function
$\rho=(\rho(C_1),\ldots,\rho(C_r))\in {\mathcal P}( G_*)$ with
$$
\|\rho\|=\sum_{1\leq i\leq r}|\rho(C_i)|\leq n \text{ and }
\rho(C_i)=(\mu_1^i,\ldots,\mu_{b_i}^i)
$$
for $b_i\in\mathbb{Z}_+$ with $1\leq i\leq r$, we introduce the
element $C_n^{\rho}\in\mathbb{F}G_n$ as
$$
C_n^{\rho}=\sum_{T}C^{\rho}_{n,T},
$$
where the summation is over all subsets $T$ of $\mathbb{N}_n$ with
cardinality $\|\rho\|$ and $C^{\rho}_{n,T}$ is defined as follows:
\begin{equation}\label{eqn:CTrho}
C^{\rho}_{n,T}=\sum(g,(k^1_{1,1},\ldots,k^1_{1,\mu^1_1})\cdots
(k^1_{b_1,1},\ldots,k^1_{b_1,\mu^1_{b_1}})
\cdots(k^r_{1,1},\ldots,k^r_{1,\mu^r_1}) \cdots
(k^r_{b_r,1},\ldots,l^r_{b_r,\mu^r_{b_r}}))
\end{equation}
where the sum is over the sequences
$k^1_{1,1},\ldots,k^1_{1,\mu^1_1};\ldots
;k^1_{b_1,1},\ldots,k^1_{b_1,\mu^1_{b_1}}; \ldots; k^r_{1,1},
\ldots,k^r_{1,\mu^r_1};\quad $ $\ldots ;k^r_{b_r,1},\ldots$,
$k^r_{b_r,\mu^r_{b_r}}$ of pairwise distinct indices taken from
$T$, $g=(g_1,\ldots,g_n)\in G^n$ such that the cycle product
corresponding to each cycle $(k^t_{s,1},\ldots, k^t_{s,\mu^t_s})$
lies in $C_t$ for $1\leq t\leq r, 1\leq s\leq b_t$ and $g_{j}=1$
for all $j\notin T$. Here $(k^t_{s,1},\ldots, k^t_{s,\mu^t_s})$ is
understood as a cycle in the symmetric group $S_n$. For the empty
partition-valued function $\emptyset$, we set $C_n^{\emptyset}=1$.
For $\rho\in\mathcal{P}(G_*)$ with $\|\rho\|>n$, we set
$C_n^{\rho}=0$.
\begin{example}\label{eg1}
Suppose $G=\mathbb{Z}_2=\{1,-1\}$ and $\rho\in\mathcal{P}(G_*)$
satisfying $\rho(1)=(1)$ and $\rho(-1)=(2)$. Then it can be
checked that
\begin{align*} C_3^{\rho}=&~2((1,1,-1),(1)(23))+2((1,-1,1),(1)(23)) +
2((1,1,-1),(2)(13))\\
&+2((-1,1,1),(1)(13))+2((1,-1,1),(3)(12))+2((-1,1,1),(3)(12)).
\end{align*}
\end{example}
\begin{rem} In the case when $G=\{1\}$,
$\mathcal{P}(G_*)$ coincides with the set of all partitions and
the elements $C^{\rho}_n$ coincides with $C_n^{\rho(C_1)}$ defined
in ~\cite[Section 4, (4.3)]{MO}.
\end{rem}
Let $\rho_1^1=((1),\emptyset\ldots,\emptyset)\in\mathcal{P}(G_*)$
be the partition-valued function with $\rho^1_1(C_1)=(1)$ and $
\rho^1_1(C_k)=\emptyset$ for $2\leq k\leq r$. Note that
$C_n^{\rho^1_1}=n1$. Given two partition-valued functions
$\rho=(\rho(C_1),\ldots,\rho(C_r))$ and
$\pi=(\pi(C_1),\ldots,\pi(C_r))$, we set
$\rho\cup\pi=(\rho(C_1)\cup\pi(C_1),\ldots,\rho(C_r)\cup\pi(C_r))$,
where $\rho(C_k)\cup\pi(C_k)$ is the partition whose parts are
those of $\rho(C_k)$ and $\pi(C_k)$ in decreasing order for
$1\leq k\leq r$. Observe that for any $\rho\in\mathcal{P}(G_*)$
with $\|\rho\|\leq n$, we have
\begin{equation}\label{eqn:reln for ccs}
C_n^{\rho\cup\rho^1_1}=(n-\|\rho\|)C_n^{\rho}.
\end{equation}
By definition~(\ref{defn: degree}), each $\be\in G_n$ appearing on
the right hand side of the equation~(\ref{eqn:CTrho}) has degree
equal to $\sum_{1\leq j\leq b_1, \mu_j^1\geq 2}\mu_j^1+\sum_{k\geq
2}|\rho(C_k)|$, and hence
$$
\text{deg}(C_n^{\rho})=\sum_{1\leq j\leq b_1, \mu_j^1\geq
2}\mu_j^1+\sum_{k\geq 2}|\rho(C_k)|.
$$
\begin{prop}\label{prop:basis for mZ(Gn)}
The center $\mZ_{0,n}$ of the algebra $\mathbb{F}G_n$ has two
bases
\begin{enumerate}
\item $C_n^{\rho}, \quad \|\rho\|=n,$ \label{eln:conj.class.sum0}

\item $C_n^{\rho}, \quad \|\rho\|\leq n, \text{ and }\rho(C_1)
\text{ has no part equal to 1 }.$ \label{eln:conj.class.sum1}
\end{enumerate}
Moreover, the elements $C_n^{\rho}$ of degree less than or equal
to $k$ of each basis
 form a basis of $\mZ_{0,n}^k$ for $0\leq
k\leq n$.
\end{prop}
\begin{proof}
Note that $C_n^{\rho}$ with $\|\rho\|=n$ is proportional to the
conjugacy class sum corresponding to the type $\rho$. Hence the
elements~(\ref{eln:conj.class.sum0}) form a basis of $\mZ_{0,n}$.
The second statement follows from the fact that  the elements of
type~(\ref{eln:conj.class.sum1}) are proportional to those of
type~(\ref{eln:conj.class.sum0}) by~(\ref{eqn:reln for ccs}).
\end{proof}
\begin{lem}\label{lem:mult.for C}
Let $\rho=(\rho(C_1),\ldots,\rho(C_r))$ and
$\pi=(\pi(C_1,\ldots,\pi(C_r)))$ be two partition-valued functions
on $G_*$ satisfying $\rho(C_1)$ and $\pi(C_1)$ have no part equal
to 1, and let $\|\rho\|+\|\pi\|\leq n$. Then
\begin{equation}\label{eqn:mult.formula}
C_n^{\rho}C_n^{\pi}=C_n^{\rho\cup\pi}+(\ldots),
\end{equation}
where $(\ldots)$ stands for a linear combination of the elements $C_n^{\varrho}$
with $\|\varrho\|<\|\rho\|+\|\pi\|$.
\end{lem}
\begin{proof}
By~(\ref{eqn:CTrho}), we have
\begin{equation}\label{eqn:mult. for Cnrho}
C_n^{\rho}C_n^{\pi}=\sum_{T\cap
S=\emptyset}C^{\rho}_{n,T}C^{\pi}_{n,S}+\sum_{T\cap
S\neq\emptyset}C^{\rho}_{n,T}C^{\pi}_{n,S},
\end{equation}
where the sum is over subsets $T$ and $S$ of $\mathbb{N}_n$ with
cardinality $\|\rho\|$ and $\|\pi\|$, respectively. Note that the
first sum on the right hand side of~(\ref{eqn:mult. for Cnrho})
equals to $C_n^{\rho\cup\pi}$. Let $\be=(g,\sigma),
\be^{\prime}=(g^{\prime},\sigma^{\prime})\in G_n$ be elements
occurring in the expansion of $C_{n,T}^{\rho}$ and
$C_{n,S}^{\pi}$, respectively. Since both $\rho(C_1)$ and
$\pi(C_1)$ have no parts equal to 1, then $J(\be)=T$ and
$J(\be^{\prime})=S$ and hence $J(\be\be^{\prime})\subseteq T\cup
S$. Therefore each element $\be^{''}\in G_n$ occurring in the
expansion of $\sum_{T\cap
S\neq\emptyset}C^{\rho}_{n,T}C^{\pi}_{n,S}$ is of degree less than
$\|\rho\|+\|\pi\|$ and hence $\sum_{T\cap
S\neq\emptyset}C^{\rho}_{n,T}C^{\pi}_{n,S}$ is a linear
combination of the elements $C_n^{\varrho}$ with
$\|\varrho\|<\|\rho\|+\|\pi\|$.
\end{proof}
For each $1\leq k\leq r$ and integer $i\in\mathbb{Z}_+$, set
$\rho^k_i\in\mathcal{P}(G_*)$ as follows:
$$
\rho^k_i=(\emptyset,\ldots,(i),\ldots,\emptyset),
$$
with $\rho^k_i(C_k)=(i)$ and $\rho^k_i(C_l)=\emptyset$ for $l\neq
k.$
Observe that each $\rho\in\mathcal{P}(G_*)$ can be written as a
union of these $\rho_i^k$. Then the following is obvious by
Lemma~\ref{lem:mult.for C}.
\begin{cor}
Let
$l=(l_2^1,\ldots,l_n^1,l^2_1,\ldots,l^2_n,\ldots,l^r_1,\ldots,l^r_n)$
run over the $(nr-1)$-tuples of nonnegative integers such that
$\sum_{k=2}^rl_1^k+\sum_{i=2}^n\sum_{k=1}^ril_i^k\leq n$. Then the
monomials
\begin{equation}\label{monomials1}
\left((C_n^{\rho^1_2})^{l^1_2}\cdots(C_n^{\rho^1_n})^{l^1_n}\right)\left(
(C_n^{\rho^2_1})^{l^2_1}\cdots(C_n^{\rho^2_n})^{l^2_n}\right)
\cdots
\left((C_n^{\rho^r_1})^{l^r_1}\cdots(C_n^{\rho^r_n})^{l^r_n}\right)
\end{equation}
form a basis of $\mZ_{0,n}$. Moreover, for any $0\leq t\leq n$,
the monomials~(\ref{monomials1}) with
$\sum_{k=2}^rl_1^k+\sum_{i=2}^n\sum_{k=1}^ril_i^k\leq t$ form a
basis of $\mZ_{0,n}^t$.
\end{cor}
\subsection{ A basis of $\overline{\mZ}_0$.}
For any subset
$T\subseteq \mathbb{N}_n$, set
\begin{align}
\va_{T}&=\prod_{i\in T}\va_{i}\label{eqn:set-va1},\\
\overline{\va}_T&=\prod_{i\in T}(1-\va_{i})\label{eqn:set-va2}.
\end{align}
For any $\rho=(\rho(C_1),\ldots,\rho(C_r))\in\mathcal{P}(G_*)$ with $\|\rho\|\leq n$,
set
\begin{equation}\label{defn:Delta_nrho}
\Delta_n^{\rho}=\sum_{T}C^{\rho}_{n,T}\overline{\va}_{T},
\end{equation}
where the summation is over all subsets $T\subseteq\mathbb{N}_n$
of cardinality $\|\rho\|$. For $\rho\in\mathcal{P}(G_*)$ with
$\|\rho\|>n$, we set $\Delta_n^{\rho}=0$. In particular,
\begin{equation}
\Delta_n^{\rho^1_1}=\sum_{i=1}^n(1-\va_i).
\end{equation}

By~(\ref{eqn:semigprel}), we have
$C^{\rho}_{n,T}\overline{\va}_{T}=\overline{\va}_{T}C^{\rho}_{n,T}$, and hence
$$
\Delta_n^{\rho}=\sum_{T}\overline{\va}_{T}C^{\rho}_{n,T}=
\sum_{T}\overline{\va}_{T}C^{\rho}_{n,T}\overline{\va}_{T},
$$
where the summation is over all subsets $T\subseteq\mathbb{N}_n$
of cardinality $\|\rho\|$.
\begin{example} Retain the assumption in Example~\ref{eg1}, we
have
$$
\Delta_3^{\rho}=C_3^{\rho}(1-\va_1)(1-\va_2)(1-\va_3).
$$

\end{example}
\begin{lem}\label{lem:Del in ovmZ0n}
The element $\Delta_n^{\rho}$ lies in $\overline{\mZ}_{0,n}$ for
any $\rho\in\mathcal{P}(G_*)$.
\end{lem}
\begin{proof}
By~(\ref{eqn:semigprel}) and the fact that any two elements of
$G_n$ occurring in $C_n^{\rho}$ are conjugate to each other, it
can be easily checked that $\Delta_n^{\rho}$ commutes with $G_n$.
Since $\ovG_n$ is generated by the wreath product $G_n$ and the
pairwise commuting idempotents $\va_1,\ldots,\va_n$, it suffices
to show that $\Delta_n^{\rho}$ commutes with $\va_i$ for each $1\leq
i\leq n$. Note that for $\va_i$ and any subset
$T\subseteq\mathbb{N}_n$ with cardinality $\|\rho\|$, if $i\in T$,
then $\va_i\overline{\va}_{T}=0=\overline{\va}_{T}\va_i$ since
$\va_i(1-\va_i)=0$; otherwise,
$\va_iC^{\rho}_{n,T}=C^{\rho}_{n,T}\va_i$. Therefore
$\va_iC^{\rho}_{n,T}\overline{\va}_{T}=C^{\rho}_{n,T}\overline{\va}_{T}\va_i$.
This implies $\va_i\Delta_n^{\rho}=\Delta_n^{\rho}\va_i$ for each
$1\leq i\leq n$.
\end{proof}
\begin{prop}\label{prop:Im ofDel} For each $n$ and $\rho\in\mathcal{P}(G_*)$, we have
\begin{equation}
\theta_n(\Delta_n^{\rho})=\Delta_{n-1}^{\rho}.
\end{equation}
\end{prop}
\begin{proof}
Recall that $\theta_n(\Delta_n^{\rho})$ can be identified with
$\va_n\Delta_n^{\rho}\va_n$, which reduces to remove
from~(\ref{defn:Delta_nrho}) all summands which correspond to the
subsets $T$ containing $n$. If $\|\rho\|=n$, then all the summands
vanish. Hence we have
$$
\theta_n(\Delta_n^{\rho})=\sum_{T\subseteq\mathbb{N}_{n-1},|T|=\|\rho\|}
C^{\rho}_{n,T}\overline{\va}_{T}=\Delta_{n-1}^{\rho}.
$$
\end{proof}
For fixed $\rho\in\mathcal{P}(G_*)$, the degrees of  elements
$\Delta_n^{\rho}$ for all $n$ equal to $\|\rho\|$, hence we have
the following.
\begin{cor}\label{cor:Deltarho}
For any $\rho\in\mathcal{P}(G_*)$, there exists an element
$\Delta^{\rho}\in \overline{\mZ}_0$ such that
$$
\theta^n(\Delta^{\rho})=\Delta_n^{\rho}
$$
for any $n\geq 0$.
\end{cor}

Let $I(n,G)=\F\ovG_n(1-\va_n)$ be the left ideal of $\F\ovG_n$
generated by $1-\va_n$.
\begin{lem}\label{lem:In(G)and mZ0(Gn)}
For any positive integer $n$, we have
$$I(n,G)\cap \overline{\mZ}_{0,n}=\mZ_{0,n}(1-\va_1)\cdots(1-\va_n).$$
\end{lem}
\begin{proof}
Suppose $x\in I(n,G)\cap \overline{\mZ}_{0,n}$, then $x\va_n=0$ as
$(1-\va_n)\va_n=0$. Since $x\in\overline{\mZ}_{0,n}$, $x$ is
invariant under the conjugation by the elements of $G_n$ and hence
$x\va_k=0$ for $1\leq k\leq n$ by (\ref{eqn:semigprel}). Thus
$x=x(1-\va_{1})\cdots(1-\va_n)$. Note that we can write $x$ as
\begin{equation}\label{eqn:exp.of x }
 x=y+y^{\prime},
\end{equation} where $y$ and $y^{\prime}$ are spanned by elements of
$G_n$ and $\ovG_n\setminus G_n$, respectively. Moreover $y$ and
$y^{\prime}$ are uniquely determined by $x$. Since for each
$\ga\in\ovG_n\setminus G_n$ there exists $1\leq i\leq n$ such that
$\ga\va_i=\ga$, each element in $\ovG_n\setminus G_n$ is
annihilated by $(1-\va_{1})\cdots(1-\va_n)$ and hence
$y^{\prime}(1-\va_{1})\cdots(1-\va_n)=0$. This implies
\begin{equation}\label{eqn:exp.2of x}
x=y(1-\va_{1})\cdots(1-\va_n).
 \end{equation} Now for any $\be\in
G_n$, $x=\be x\be^{-1}=\be y\be^{-1}(1-\va_{1})\cdots(1-\va_n)$
since
$\be(1-\va_{1})\cdots(1-\va_n)=(1-\va_{1})\cdots(1-\va_n)\be$
by~(\ref{reln:symm.va.}) and~(\ref{reln:gp.va}). Hence we have
$$x=\be y\be^{-1}+\sum_{\emptyset\neq
S\subseteq\mathbb{N}_n}(-1)^{|S|}\be y\be^{-1}\va_S.
$$
Clearly $\sum_{\emptyset\neq S\subseteq\mathbb{N}_n}(-1)^{|S|}\be
y\be^{-1}\va_S$ is a linear combination of elements of
$\ovG_n\setminus G_n$ while $\be y\be^{-1}\in\mathbb{F}G_n$.
Comparing the two expansions (\ref{eqn:exp.of x }) and
(\ref{eqn:exp.2of x}) of $x$ we obtain $\be y\be^{-1}=y$ for each
$\be\in G_n$. This means $y\in \mZ_{0,n}$, and hence $x\in
\mZ_{0,n}(1-\va_1)\cdots(1-\va_n)$.

Conversely, suppose $x=y(1-\va_{1})\cdots(1-\va_n)$ for some $y\in
\overline{\mZ}_{0,n}$. It is clear that $x\in I(n,G)$. On the
other hand, since $y\in \overline{\mZ}_{0,n}$, $x$ is invariant
under the conjugation by the elements of $G_n$. Observe that $x$
is both left and right annihilated by $\va_{1},\ldots,\va_n$. This
implies $x$ commutes with $\ovG_n$, and hence $x\in
\overline{\mZ}_{0,n}$.
\end{proof}
Recall that for any $\be\in G_n$, $\bar{J}(\be)=\{i|1\leq i\leq n,
\be_{ii}=1\}$ and for any subset $T\subseteq\mathbb{N}_n$,
$\va_{T}=\prod_{i\in T}\va_i$.
\begin{lem}\label{lem:bijection1}
The mapping
\begin{equation}\label{map:StoT}
\be\mapsto \be\va_{\bar{J}(\be)}
\end{equation}
defines a bijection of $G_n$ onto the set of all elements of
degree $n$ in $\ovG_n$ satisfying that the $i$th row and $i$th
column are zero or nonzero at the same time for any $1\leq i\leq
n$.
\end{lem}
\begin{proof}
Note that $\be\va_{\bar{J}(\be)}$ is the matrix obtained from
$\be$ by replacing all the $1$'s on the diagonal by zeros. This
implies that for any $1\leq i\leq n$, the $i$th row and the $i$th
column of $\be\va_{\bar{J}(\be)}$ are zero or nonzero at the same
time and moreover $\text{deg}(\be\va_{\bar{J}(\be)})=n.$

Conversely, suppose $\ga\in\ovG_n$ is of degree $n$ satisfying
that the $i$th row and $i$th column are zero or nonzero at the
same time for each $1\leq i\leq n$. Let $\be\in G_n$ be defined as
follows:
\begin{eqnarray}
\be_{ij}=\left\{
\begin{array}{ll}
1,\quad \text{if } i=j \text{ and the }i\text{-th row is zero}\\
\gamma_{ij},\quad \text{otherwise}
\end{array}\right.
\end{eqnarray}
It is easy to see that $\ga$ is the image of $\be$ under the
mapping~(\ref{map:StoT}).
\end{proof}
\begin{lem}\label{lem:injective theta}
The restriction of the projection $\theta_n:
\overline{\mZ}_{0,n}\rightarrow \overline{\mZ}_{0,n-1}$ to the
subspace $\overline{\mZ}_{0,n}^{n-1}$ is injective.
\end{lem}
\begin{proof}
Let $x\in \overline{\mZ}_{0,n}$ and $\theta_n(x)=0$. It suffices
to show that $\text{deg}(x)=n$ unless $x=0$. Note that
$\theta_n(x)=0$ implies that $x\va_n=0$, and hence
$x=x(1-\va_n)\in I(n,G)$, which means $x\in I(n,G)\cap
\overline{\mZ}_{0,n}$. By Lemma~\ref{lem:In(G)and mZ0(Gn)}, $x$
can be written as a linear combination of elements of the form
$\be(1-\va_1)\cdots(1-\va_n)$ for some $\be\in G_n$. Let us
rewrite as
\begin{align}
\be(1-\va_1)\cdots(1-\va_n)&=\sum_{Q\subseteq\mathbb{N}_n}(-1)^{|Q|}\be\va_{Q}\notag\\
&=\sum_{Q\supseteq
\bar{J}(\be)}(-1)^{|Q|}\be\va_{Q}+\sum_{Q\nsupseteq
\bar{J}(\be)}(-1)^{|Q|}\be\va_{Q}\label{decomp. for sum}.
\end{align}
Note that all elements appearing in the first sum in~(\ref{decomp.
for sum}) are of degree $n$ while those appearing in the second
sum are of degree strictly less than $n$. Now it reduces to show
that the elements $\Phi(\be):=\sum_{Q\supseteq
\bar{J}(\be)}(-1)^{|Q|}\be\va_{Q}$ are linearly independent.
Suppose $\sum_{\be\in G_n}d_{\be}\Phi(\be)=0$ for some
$d_{\be}\in\mathbb{F}$. We shall show that $d_{\be}=0$ for all
$\be\in G_n.$ Recall from~(\ref{defn:rank}) that for each
$\ga\in\ovG_n$, $\text{rank}(\ga)=\{l|~1\leq l\leq n, \ga_{kl}\neq
0 \text{ for some }1\leq k\leq n\}$, that is, the number of
nonzero columns in $\ga$. Then we have
$\text{rank}(\be\va_Q)=n-|Q|$, and hence
\begin{align}
\text{rank}(\be\va_{Q})&=n-|\bar{J}(\be)|, \quad \text{if }Q=\bar{J}(\be)\label{eqn:rank1}\\
\text{rank}(\be\va_{Q})&<n-|\bar{J}(\be)|, \quad \text{if
}Q\supset \bar{J}(\be)\label{eqn:rank2}.
\end{align}
This implies that the rank of the elements appearing in
$\Phi(\be)$ is no bigger than $n-|\bar{J}(\be)|$ For each $0\leq
k\leq n$, set
$$
\Psi_k=\{\be\in G_n|~ n-|\bar{J}(\be)|=k\}.
$$
Then $G_n$ is the disjoint union as:
$$
G_n=\cup_{k=0}^n\Psi_k.
$$
Hence we can rewrite $\sum_{\be\in G_n}d_{\be}\Phi(\be)=0$ as
\begin{equation}\label{eqn:n+1 summands}
\sum_{\be\in\Psi_0}d_{\be}\Phi(\be)+\cdots+
\sum_{\be\in\Psi_{n-1}}d_{\be}\Phi(\be)+
\sum_{\be\in\Psi_n}d_{\be}\Phi(\be)=0.
\end{equation}
Note that the elements of rank $n$ in~(\ref{eqn:n+1 summands}) are
of the form $\be\va_{\bar{J}(\be)}$ with $\be\in\Psi_n$
by~(\ref{eqn:rank1}) and~(\ref{eqn:rank2}) and hence
$$
\sum_{\be\in\Psi_n}d_{\be}\be\va_{\bar{J}(\be)}=0.
$$ But Lemma~\ref{lem:bijection1} says that $\be\va_{\bar{J}(\be)}$ with
$\be\in G_n$ are pairwise distinct elements in $\ovG_n$, and hence
$$d_{\be}=0 \text{ for all }\be\in\Psi_{n}.
$$  Then~(\ref{eqn:n+1
summands}) becomes to
$\sum_{\be\in\Psi_0}d_{\be}\Phi(\be)+\cdots+\sum_{\be\in\Psi_{n-1}}d_{\be}\Phi(\be)=0$.
Apply the same argument to the elements of rank $n-1$ and we get
$d_{\be}=0$ for all $\be\in\Psi_{n-1}$. Continue this way and
finally we get  $d_{\be}=0$ for any $\be\in G_n$.
\end{proof}

\begin{prop}\label{prop:basis for ovmZ0}
For any positive integer $n$, the elements $\Delta_n^{\rho}$,
where $\rho\in\mathcal{P}(G_*)$ with $\|\rho\|\leq n$, form a
basis of $\overline{\mZ}_{0,n}$. Furthermore, for $0\leq k\leq n$,
the elements $\Delta_n^{\rho}$ with $\|\rho\|\leq k$ form a basis
of $\overline{\mZ}^k_{0,n}$.
\end{prop}
\begin{proof}
It suffices to show the second claim of the proposition. We shall
prove this using the induction on $n$. In the case $n=1$, note
that $\ovG_n$ is generated by $G$ and $\va_1$ subject to the
relations $h\va_1=\va_1h=\va_1$ and $\va_1^2=\va_1$ for any $h\in
G$. Hence it can be easily checked that the center
$\overline{\mZ}_{0,1}$ of $\mathbb{F}\ovG_1$ is the subalgebra
generated by the center of $\mathbb{F}G$ and $\va_1$, and then the
proposition follows.

Assume that $n\geq 2$ and $k\leq n-1$. By the induction
hypothesis, the elements $\Delta_{n-1}^{\rho}$ with $\|\rho\|\leq
k$ form a basis of $\overline{\mZ}^k_{0,n-1}$.
Proposition~\ref{prop:Im ofDel} says that the image of
$\Delta_n^{\rho}$ is $\Delta_{n-1}^{\rho}$ for all
$\rho\in\mathcal{P}(G_*)$ with $\|\rho\|<n$. But
Lemma~\ref{lem:injective theta} implies that the restriction of
$\theta_n$ to $\overline{\mZ}^k_{0,n}$ is injective. Therefore
$\Delta_n^{\rho}$ with $\|\rho\|\leq k$ form a basis of
 $\overline{\mZ}^k_{0,n}$.

Next, we shall first show that
\begin{equation}\label{eqn:decomp.of A0(n,G)}
\overline{\mZ}_{0,n}=\overline{\mZ}^{n-1}_{0,n}\oplus (I(n,G)\cap
\overline{\mZ}_{0,n})
\end{equation}
and then the proposition will follow from the claim that the
elements $\Delta_n^{\rho}$ with $\|\rho\|=n$ form a basis of
$I(n,G)\cap \overline{\mZ}_{0,n}$.

Consider the restriction
$\theta_n\downarrow_{\overline{\mZ}_{0,n}}$ of $\theta_n$ to
$\overline{\mZ}_{0,n}$. It follows from the proof of
Lemma~\ref{lem:In(G)and mZ0(Gn)} that the kernel of
$\theta_n\downarrow_{\overline{\mZ}_{0,n}}$ is $I(n,G)\cap
\overline{\mZ}_{0,n}$ and the image is contained in
$\overline{\mZ}_{0,n-1}$. But clearly $\theta_n$ maps
$\overline{\mZ}^{n-1}_{0,n}$ onto
$\overline{\mZ}^{n-1}_{0,n-1}=\overline{\mZ}_{0,n-1}$ as shown
above, and hence we obtain
$$
\overline{\mZ}_{0,n}=\overline{\mZ}^{n-1}_{0,n}+ (I(n,G)\cap
\overline{\mZ}_{0,n})
$$
But Lemma~\ref{lem:injective theta} implies that
$\overline{\mZ}^{n-1}_{0,n}\cap (I(n,G)\cap
\overline{\mZ}_{0,n})=\{0\}$, and then ~(\ref{eqn:decomp.of
A0(n,G)}) follows.

Finally, let us show that the elements $\Delta_n^{\rho}$ with
$\|\rho\|=n$ form a basis of $I(n,G)\cap \overline{\mZ}_{0,n}$.
Note that for each $\rho\in\mathcal{P}(G_*)$ with $\|\rho\|=n$
$$
\Delta_n^{\rho}=C_n^{\rho}(1-\va_1)\cdots(1-\va_n).
$$
Proposition~\ref{prop:basis for mZ(Gn)} says that $C_n^{\rho}$
with $\|\rho\|=n$ form a basis of $\overline{\mZ}_{0,n}$, and
hence $I(n,G)\cap \overline{\mZ}_{0,n}$ is spanned by
$\Delta_n^{\rho}$ with $\|\rho\|=n$ by Lemma~\ref{lem:In(G)and
mZ0(Gn)}. It remains to show that these elements are linearly
independent. Now suppose $\sum_{\rho,
\|\rho\|=n}d_{\rho}\Delta_n^{\rho}=0$ for some
$d_{\rho}\in\mathbb{F}$. Note that the difference
$\Delta_n^{\rho}-C_n^{\rho}$ is spanned by elements from
$\ovG_n\setminus G_n$ while $C_n^{\rho}\in\mathbb{F}G_n$. This
implies $\sum_{\rho, \|\rho\|=n}d_{\rho}C_n^{\rho}=0$, and hence
$d_{\rho}=0$ for all $\rho\in\mathcal{P}(G_*)$ with $\|\rho\|=n$.
\end{proof}
Recall from Corollary~\ref{cor:Deltarho} the definition of
elements $\Delta^{\rho}$ for each $\rho\in\mathcal{P}(G_*)$. The
next theorem follows from Proposition~\ref{prop:basis for ovmZ0}.
\begin{thm}\label{thm:basis for vmZ0}
The elements $\Delta^{\rho},\rho\in\mathcal{P}(G_*)$ form a basis
of the algebra $\overline{\mZ}_0$. Moreover, for any $k\geq 0$,
the elements $\Delta^{\rho}$ with $\|\rho\|\leq k$ form a basis of the
$k$th subspace $\overline{\mZ}_0^k$.
\end{thm}
\begin{lem}\label{lem:mult.for Delta}
For any partition-valued functions $\rho, \pi\in\mathcal{P}(G_*)$
with $\|\rho\|+\|\pi\|\leq n$, the following holds in
$\overline{\mZ}_{0,n}$:
$$
\Delta_n^{\rho}\Delta_n^{\pi}=\Delta_n^{\rho\cup\pi}+(\ldots),
$$
where $(\ldots)$ is a linear combination of elements
$\Delta_n^{\varrho}$ with $\|\varrho\|<\|\rho\|+\|\pi\|$.
\end{lem}
\begin{proof}
By~(\ref{defn:Delta_nrho}), we have
\begin{equation}\label{eqn: decom.of mult for Delta}
\Delta_n^{\rho}\Delta_n^{\pi}=\sum_{T\cap S=\emptyset}
C^{\rho}_{n,T}\overline{\va}_TC^{\pi}_{n,S}\overline{\va}_S+
\sum_{T\cap
S\neq\emptyset}C^{\rho}_{n,T}\overline{\va}_TC^{\pi}_{n,S}\overline{\va}_S,
\end{equation}
where $T$ and $S$ are subsets of $\mathbb{N}_n$ of order
$\|\rho\|$ and $\|\pi\|$, respectively. Note that the first sum on
the right hand side of (\ref{eqn: decom.of mult for Delta}) is
$\Delta_n^{\rho\cup\pi}$ whereas the second sum is of degree
strictly less than $\|\rho\|+\|\pi\|$.
\end{proof}

Clearly Theorem~\ref{thm:basis for vmZ0} and
Lemma~\ref{lem:mult.for Delta} give rise to the following.
\begin{cor}
Let
$l=(l^1_1,\ldots,l^1_n,l^2_1,\ldots,l^2_n,\ldots,l^r_1,\ldots,l^r_n)$
run over the $nr$-tuples of nonnegative integers such that
$\sum_{k=1}^rl_1^k+\sum_{k=1}^r2l_2^k+\cdots+\sum_{k=1}^rnl_n^k\leq
n$. Then the monomials
\begin{equation}\label{monomials2}
\left((\Delta_n^{\rho^1_1})^{l^1_1}(\Delta_n^{\rho^2_1})^{l^2_1}\cdots(\Delta_n^{\rho^r_1})^{l^r_1}\right)
\cdots
\left((\Delta_n^{\rho^1_n})^{l^1_n}(\Delta_n^{\rho^2_n})^{l^2_n}\cdots(\Delta_n^{\rho^r_n})^{l^r_n}\right)
\end{equation}
form a basis of $\overline{\mZ}_{0,n}$. Moreover, for any $0\leq
t\leq n$, the monomials~(\ref{monomials2}) with
$\sum_{k=1}^rl_1^k+\sum_{k=1}^r2l_2^k+\cdots+\sum_{k=1}^rnl_n^k\leq
t$ form a basis of $\overline{\mZ}_{0,n}^t$.
\end{cor}

\begin{cor}
The monomials
\begin{equation}\label{eqn:monomials2}
\left((\Delta^{\rho^1_1})^{l^1_1}(\Delta^{\rho^2_1})^{l^2_1}\cdots(\Delta^{\rho^r_1})^{l^r_1}\right)
\left((\Delta^{\rho^1_2})^{l^1_2}(\Delta^{\rho^2_2})^{l^2_2}\cdots(\Delta^{\rho^r_2})^{l^r_2}\right)
\cdots
\end{equation} for $l_i^k\in\mathbb{Z}_+$ with
$\sum_{i\geq 1}\sum_{k=1}^rl_i^k<\infty$ form a basis of the
algebra $\overline{\mZ}_0$. Moreover for each $t\geq 0$, the
monomials~(\ref{eqn:monomials2}) with $\sum_{i\geq
1}\sum_{k=1}^rl_i^k\leq t$ form a basis of the subspace
$\overline{\mZ}_0^t$.
\end{cor}
\section{The structure of the algebras $\overline{\mZ}_m, m>0$}\label{mZm}
In this section we shall study the algebraic structure of the
algebra $\overline{\mZ}_m$ for $m>0$. Let us first introduce a
subalgebra of $\F\ovG_n$ which plays the similar role as the
center $\mZ_{0,n}$ of $\F G_n$ in the case $m=0$. Set
\begin{equation}\label{defn;Ga(m,n,G)}
\ovG_{m,n}=\{\ga\in\ovG_n|~\text{for }m+1\leq i\leq n, \text{
there exist } 1\leq j, k\leq n \text{ such that }\ga_{ij}\neq 0,
\ga_{ki}\neq 0\}
\end{equation}
That is, $\ovG_{m,n}$ consists of all elements whose $i$th row and
$i$th column are nonzero for each $m+1\leq i\leq n$. Denote by
$\mathbb{F}\ovG_{m,n}$ the subspace of $\F\ovG_n$ spanned by
$\ovG_{m,n}$. Recall that $G^{\prime}_{n-m}$ is the subgroup of
$G_n$ consisting of matrices whose first $m$ diagonal elements are
$1$. Set $\mZ^*_{m,n}$ to be the subspace of
$\mathbb{F}\ovG_{m,n}$ formed by the elements invariant under the
conjugation by the elements of the group $G^{\prime}_{n-m}$, that
is, we have
$$
\mZ^*_{m,n}=(\F\ovG_{m,n})^{G^{\prime}_{n-m}}.
$$

\subsection{Bases of $\mZ^*_{m,n}$}
Throughout the section we assume $0< m\leq n$. Let us begin with
introducing another filtration on the algebra $\F\ovG_n$. For
$\ga\in\ovG_n$, set
\begin{align}
J_m(\ga)&=\{k|~m+1\leq k\leq n, \ga_{kk}\neq 1\},\notag\\
\bar{J}_m(\ga)&=\{k|~m+1\leq k\leq n, \ga_{kk}=1\}=\{m+1,\cdots,n\}\setminus J_m(\ga),\label{defn:barJ-m}\\
\text{deg}_m(\ga)&=|J_m(\ga)|,\label{defn: degree-m}
\end{align}
We will call $\text{deg}_m(\ga)$ the $m$-degree of $\ga$.

\begin{lem}
For $\ga,\ga^{\prime}\in\ovG_n$, we have
\begin{equation}\label{eqn:mult.formula for degree 2}
\text{deg}_m(\ga\ga^{\prime})\leq
\text{deg}_m(\ga)+\text{deg}_m(\ga^{\prime}).
\end{equation}
\end{lem}
\begin{proof}
For any $m+1\leq k\leq n$, if $k\in \bar{J}_m(\ga)\cap
\bar{J}_m(\ga^{\prime})$ then $k\in \bar{J}_m(\ga\ga^{\prime})$.
Hence $J_m(\ga\ga^{\prime})\subseteq J_m(\ga)\cup
J_m(\ga^{\prime})$ and (\ref{eqn:mult.formula for degree 2})
follows.
\end{proof}

For each $0\leq k\leq n-m$, set $\Ga_m^k(\F\ovG_n)$ to be the
subspace of $\F\ovG_n$ spanned by the elements $\ga\in\ovG_n$ of
$m$-degree less than or equal to $k$. Then we get a new filtration
of $\F\ovG_n$:
\begin{equation}\label{m-filtration}
\F\ovG_m=\Ga_m^0(\F\ovG_n)\subseteq
\Ga_m^1(\F\ovG_n)\subseteq\cdots\subseteq
\Ga_m^{n-m}(\F\ovG_n)=\F\ovG_n
\end{equation}
For any subspace $S$ of $\F\ovG_n$ we will use the symbol
$\Ga_m^k(S)$ to indicate the $k$th term of the induced filtration.

Next we shall classify the $G^{\prime}_{n-m}$-orbits by
conjugation in $\ovG_{m,n}$ in order to get bases of
$\mZ^*_{m,n}$. Note that the additive semigroup $\mathbb{Z}_+$ is
isomorphic to the free abelian semigroup $\{1,z,z^2,\ldots\}$ with
unity $1$ and one generator $z$. Consider the semigroup product
$G\times \mathbb{Z}_+=\{(g,z^k)|~g\in G,k=0,1,2,\ldots\}$ and set
$\text{ord}(g,z^k)=k$ for $(g,z^k)\in G\times \mathbb{Z}_+$.
Clearly $G$ can be regarded as a subgroup of $G\times
\mathbb{Z}_+$.  Denote by $\Ga(n,G\times
\mathbb{Z}_+)~(\text{resp.} S(n,G\times \mathbb{Z}_+))$ the set
consisting of the $n\times n$ matrices
$\Omega=(\Omega_{ij})_{n\times n}$ with entries in $(G\times
\mathbb{Z}_+)\cup\{0\}$ such that any row and column contain at
most~(resp. exactly) one nonzero entry. For any
$\Omega\in\Ga(n,G\times \mathbb{Z}_+)$, set
$$
\text{ord}(\Omega)=\sum_{i,j;\Omega_{ij}\neq
0}\text{ord}(\Omega_{ij}).
$$

\begin{prop}\label{prop:desc. of conjugate orbits}
There is a natural parametrization of the
$G_{n-m}^{\prime}$-orbits in $\ovG_{m,n}$ by the pair
$(\Omega,\rho)$ with $\Omega\in \Ga(m, G\times \mathbb{Z}_+)$ and
$\rho\in\mathcal{P}(G_*)$ satisfying
$\text{ord}(\Omega)+||\rho||=n-m$.

\end{prop}
\begin{proof}
We shall first define a map
$$
\mathcal{F}: \ovG_{m,n}\rightarrow \Ga(m, G\times
\mathbb{Z}_+)\times \mathcal{P}(G_*),\qquad \ga\mapsto
(\Omega(\ga),\rho(\ga)).
$$
For $\ga\in \ovG_{m,n}$, the $m\times m$-matrix $\Omega(\ga)$ is
given as follows.  Fix $1\leq j\leq m$, then for each $1\leq i\leq
m$ set
\begin{eqnarray}
\Omega(\gamma)_{ij}=\left\{
\begin{array}{lll}
\gamma_{ij}, \quad\text{ if } \gamma_{ij}\neq 0,\\
(\gamma_{ip_k}\gamma_{p_kp_{k-1}}\ldots \gamma_{p_1j},z^k),
\quad\text{ if there exist } m+1\leq
p_1,\ldots,p_k\leq n\notag\\
\text{ such that }\gamma_{p_1j}\neq 0, \gamma_{p_2p_1}\neq
0,\ldots,\gamma_{p_kp_{k-1}}\neq 0, \gamma_{ip_k}\neq 0,\label{P-subsets}\\
0, \quad\text{otherwise}.
\end{array}\right.
\end{eqnarray}

Observe that to any $j$ with $\gamma_{pj}\neq 0$ for some $p\in
\{m+1,\ldots,n\}$, we can assign a subset $\{p_1,p_2,\ldots,p_k\}$
of $\{m+1,\ldots,n\}$ as described above. It is clear that these
subsets are pairwise disjoint. Let $P(\gamma)$ denote their union,
then $\text{ord}(\Omega(\gamma))=|P(\gamma)|$. Further, let
$P^*(\gamma)$ be the complement of $P(\gamma)$ in
$\{m+1,\ldots,n\}$. Suppose $P^*(\gamma)=\{i_1,i_2\ldots,i_l\}$
with $i_1<i_2<\ldots <i_l$. Define
$\sigma(\ga)=(\sigma_{st})_{l\times l}$ as
$\sigma_{st}=\gamma_{i_s i_t}$ for $1\leq s,t\leq l$. Then
$\sigma\in G_l$ since the $i$th row and $i$th column of $\ga$ are
nonzero for $m+1\leq i\leq n$. Define $\rho=\rho(\gamma)$ to be
the type of $\sigma$ in $G_l$.

It can be verified that if $\ga$ and $\ga^{\prime}$ belong to the
same $G^{\prime}_{n-m}$-orbit in $\ovG_{m,n}$ then
$(\Omega(\ga),\rho(\ga))=(\Omega(\ga^{\prime}),\rho(\ga^{\prime}))$.
 Indeed, suppose $\ga^{\prime}=\be\ga\be^{-1}$ for some $\be\in
G^{\prime}_{n-m}$. Clearly for $1\leq i,j\leq m$,
$\ga_{ij}=(\be\ga\be^{-1})_{ij}$. Hence it suffices to show that
$\Omega(\ga)_{ij}=\Omega(\be\ga\be^{-1})_{ij}$ in the case when
there exist $p_1, p_2,\ldots,p_k\in \{m+1,\ldots,n\}$ such that
$\gamma_{p_1j}\neq 0, \gamma_{p_2p_1}\neq
0,\ldots,\gamma_{p_kp_{k-1}}\neq 0, \gamma_{ip_k}\neq 0$. Since
$\be\in G_{n-m}^{\prime}$, there exist $m+1\leq q_l\leq n$ such
that $\be_{q_lp_l}\neq 0$ for $1\leq l\leq k$. Then it can be easily checked that
\begin{align*}
(\be\ga\be^{-1})_{q_1j}&=\be_{q_1p_1}\ga_{p_1j}\be_{jj}=\be_{q_1p_1}\ga_{p_1j}
,\\
(\be\ga\be^{-1})_{q_lq_{l-1}}&=\be_{q_lp_l}\ga_{p_lp_{l-1}}\be_{q_{l-1}p_{l-1}}^{-1},\quad 2\leq l\leq k,\\
(\be\ga\be^{-1})_{iq_k}&=\be_{ii}\ga_{ip_k}\be_{q_kp_k}^{-1}=\ga_{ip_k}\be_{q_kp_k}^{-1}.
\end{align*}
This means
$(\Omega(\be\ga\be^{-1}))_{ij}=(\gamma_{ip_k}\gamma_{p_kp_{k-1}}\ldots
\gamma_{p_1j},z^k)=\Omega(\ga)_{ij}$. On the other hand, the above
description of conjugation $\be\ga\be^{-1}$ of $\ga\in\ovG_{m,n}$
by $\be\in G^{\prime}_{n-m}$ implies that two elements
corresponding to the same $\Omega$ and $\rho$ will be in the same
orbit.
\end{proof}


\begin{rem}\label{Gn-mP Or Gn}
The pair $(\Omega,\rho)$ in Proposition~\ref{prop:desc. of
conjugate orbits} corresponds to an element of
$G_n\subseteq\ovG_{m,n}$ if and only if $\Omega\in S(m, G\times
\mathbb{Z}_+)$. Hence by the same argument of the proof of
Proposition~\ref{prop:desc. of conjugate orbits}, we conclude that
there is a natural parametrization of the
$G_{n-m}^{\prime}$-orbits in $G_n$ by the pair $(\Omega,\rho)$
with $\Omega\in S(m, G\times \mathbb{Z}_+)$ and
$\rho\in\mathcal{P}(G_*)$ satisfying
$\text{ord}(\Omega)+||\rho||=n-m$.
\end{rem}

We shall now define analogs of the elements $C_n^{\rho}$ in
$\mZ^*_{m,n}$. Firstly, for any
$\Omega\in\Ga(m,G\times\mathbb{Z}_+)$ and any subset
$P\subseteq\{m+1,\ldots,n\}$ with $\text{ord}(\Omega)=|P|$, we set
$$
\Ga(\Omega,P)=\{\ga\in\ovG_{m,n}|~\Omega(\ga)=\Omega, P(\ga)=P,
\rho(\ga)=\underbrace{\rho^1_1\cup\ldots\cup\rho^1_1}_{n-m-|P|}\}.
$$
Then for any pair $(\Omega,\rho)$ with
$\Omega\in\Ga(m,G\times\mathbb{Z}_+)$ and
$\rho\in\mathcal{P}(G_*)$ such that
$\text{ord}(\Omega)+\|\rho\|\leq n-m$, we set
\begin{equation}\label{defn:elementC}
C_n^{\Omega,\rho}=\sum_{P,T}\sum_{\ga\in\Ga(\Omega,P)}\ga
C^{\rho}_{n,T},
\end{equation}
where $P,T$ are disjoint subsets in $\{m+1,\ldots,n\}$ such that
$$|P|=\text{ord}(\Omega),\qquad |T|=\|\rho\|.$$
\begin{lem}\label{basis for Zm(n,G)}
The centralizer $\mZ^*_{m,n}$ of $G_{n-m}^{\prime}$ in
$\mathbb{F}\ovG_{m,n}$ has two bases:
\begin{enumerate}

\item $C_n^{\Omega,\rho}, \text{ where }
\Omega\in\Ga(m,G\times\mathbb{Z}_+),\rho\in\mathcal{P}(G_*) \text{
and }\text{ord}(\Omega)+\|\rho\|=n-m;$

\item $C_n^{\Omega,\rho}, \text{ where }
\Omega\in\Ga(m,G\times\mathbb{Z}_+),\rho\in\mathcal{P}(G_*),
\text{ord}(\Omega)+\|\rho\|\leq n-m \text{ and } \rho(C_1)$ $\text{has
no part equal to }1.$
\end{enumerate}
\end{lem}
\begin{proof}
Clearly the elements $C_n^{\Omega,\rho} \text{ with }
\Omega\in\Ga(m,G\times\mathbb{Z}_+),\rho\in\mathcal{P}(G_*) \text{
and }\text{ord}(\Omega)+\|\rho\|=n-m$ are proportional to the
characteristic functions of the $G_{n-m}^{\prime}$-orbit in
$\ovG_{m,n}$ corresponding to the pairs $(\Omega,\rho)$ by
Proposition~\ref{prop:desc. of conjugate orbits}. Hence these
elements form a basis of $\mZ^*_{m,n}$. Note that
$C_n^{\Omega,\rho\cup\rho^1_1}$ is proportional to
$C_n^{\Omega,\rho}$, therefore the second type of elements also
forms a basis of $\mZ^*_{m,n}$.
\end{proof}
\subsection{The basis of $\overline{\mZ}_m$}
We now introduce analogs of the elements $\Delta_n^{\rho}$ in
$\overline{\mZ}_{m,n}$. For any pair $(\Omega,\rho)$ with
$\Omega\in\Ga(m,G\times\mathbb{Z}_+)$ and
$\rho\in\mathcal{P}(G_*)$ such that
$\text{ord}(\Omega)+\|\rho\|\leq n-m$, we set
\begin{equation}\label{defn:elementC}
\Delta_n^{\Omega,\rho}=\sum_{P,T}\sum_{\ga\in\Ga(\Omega,P)}
\overline{\va}_P\ga
C^{\rho}_{n,T}\overline{\va}_T\overline{\va}_P,
\end{equation}
where $P,T$ are disjoint subsets in $\{m+1,\ldots,n\}$ such that
$$|P|=\text{ord}(\Omega),\qquad |T|=\|\rho\|.$$
\begin{rem}\label{rem:special elements}
Note that $G$  can be identified with the subgroup $\{(g,1)|g\in
G\}\subset G\times\mathbb{Z}_+$, and hence
$\overline{G}_m\subseteq \Ga(m,G\times\mathbb{Z}_+)$. Then we
have:
\begin{enumerate}
\item If  $\Omega\in
\overline{G}_m\subseteq\Ga(m,G\times\mathbb{Z}_+)$, then
$\text{ord}(\Omega)=0$ and hence it can be checked that
$\Delta_n^{\Omega,\emptyset}=\Omega.$

\item Let $\mathbb{I}\in\Ga(m,G\times\mathbb{Z}_+)$ be the
$m\times m$ identity matrix. Then
$$
\Delta_n^{\mathbb{I},\rho}=\sum_{T\subseteq\{m+1,\ldots,n\},
|T|=\|\rho\|}C_{n,T}^{\rho}\overline{\va}_T.
$$
\end{enumerate}
\end{rem}
\begin{lem}\label{lem: image of DnOr}
For any pair $(\Omega,\rho)$ with
$\Omega\in\Ga(m,G\times\mathbb{Z}_+)$ and
$\rho\in\mathcal{P}(G_*)$ such that
$\text{ord}(\Omega)+\|\rho\|\leq n-m$, the element
$\Delta_n^{\Omega,\rho}$ lies in the algebra
$\overline{\mZ}_{m,n}$. Moreover,
$\theta_n(\Delta_n^{\Omega,\rho})=\Delta_{n-1}^{\Omega,\rho}$,
where we adopt the the convention that $\Delta_k^{\Omega,\rho}=0$
if $\text{ord}(\Omega)+\|\rho\|>k-m$.
\end{lem}
\begin{proof}
Note that $\ovG_{n-m}^{\prime}$ is generated by the subgroup
$G_{n-m}^{\prime}$ and $\va_{m+1},\ldots,\va_n$. It is easy to
check that $\Delta_n^{\Omega,\rho}$ commutes with
$G_{n-m}^{\prime}$ by Proposition~\ref{prop:desc. of conjugate
orbits}. Moreover for each $m+1\leq k\leq n$, if $k$ lies in $P$ or
$T$, then $\va_k\overline{\va}_P\ga
C^{\rho}_{n,T}\overline{\va}_T\overline{\va}_P=0=\overline{\va}_P\ga
C^{\rho}_{n,T}\overline{\va}_T\overline{\va}_P\va_k$; otherwise,
$\va_k$ commutes with $\ga$ and $C^{\rho}_{n,T}$. Therefore
$\va_k\Delta_n^{\Omega,\rho}=\Delta_n^{\Omega,\rho}\va_k$ for each
$m+1\leq k\leq n$. This implies
$\Delta_n^{\Omega,\rho}\in\overline{\mZ}_{m,n}$. By applying the
similar argument of the proof of Proposition~\ref{prop:Im ofDel},
we have
$$
\theta_n(\Delta_n^{\Omega,\rho})=\sum_{P,T}\sum_{\ga\in\Ga(\Omega,P)}\overline{\va}_P\ga
C^{\rho}_{n,T}\overline{\va}_T\overline{\va}_P,
$$
where $P,T$ are disjoint subsets in $\{m+1,\ldots,n-1\}$ such that
$|P|=\text{ord}(\Omega),|T|=\|\rho\|.$ Hence
$\theta_n(\Delta_n^{\Omega,\rho})=\Delta_{n-1}^{\Omega,\rho}$.
\end{proof}
Fix $m\in\mathbb{Z}_+$, then for any $n\geq m$ and
$\Omega\in\Ga(m, G\times\mathbb{Z}_+), \rho\in\mathcal{P}(G_*)$
with $\text{ord}(\Omega)+\|\rho\|\leq n-m$, we have
$$
\text{deg}(\Delta_n^{\Omega,\rho})\leq
\text{ord}~\Omega+\|\rho\|+m.
$$ Hence by Lemma~\ref{lem: image of DnOr} there exist elements
$\Delta^{\Omega,\rho}\in \overline{\mZ}_m$ as sequences
$$
\Delta^{\Omega,\rho}:=(\Delta_n^{\Omega,\rho}|n\geq m).
$$

\begin{lem}\label{lem:Im(n,G)}
Recall that $I(n,G)=\F\ovG_n(1-\va_n)$ is the left ideal of
$\F\ovG_n$ generated by $1-\va_n$. Then we have for $0< m\leq n$,
$$
I(n,G)\cap
\overline{\mZ}_{m,n}=(1-\va_{m+1})\cdots(1-\va_n)\mZ^*_{m,n}(1-\va_{m+1})\cdots(1-\va_n).
$$
\end{lem}
\begin{proof}
Suppose $x\in I(n,G)\cap \overline{\mZ}_{m,n}$. Since $\va_n\in
\ovG_{n-m}^{\prime}$ and $(1-\va_n)\va_n=0$, we have
$\va_nx=x\va_n=0$. By (\ref{eqn:semigprel}) and the fact that $x$
is invariant under the
 conjugation by the elements of $G_{n-m}^{\prime}$, we have
 $\va_kx=x\va_k=0$ for $m+1\leq k\leq n$.
 Thus
 $x=(1-\va_{m+1})\cdots(1-\va_n)x=x(1-\va_{m+1})\cdots(1-\va_n)$.
 Note that we can write $x=y+y^{\prime}$, where $y$ and $y^{\prime}$
are spanned by elements of $\ovG_{m,n}$ and
$\ovG_n\setminus\ovG_{m,n}$, respectively. Moreover $y$ and
$y^{\prime}$ are uniquely determined by $x$. However,
$(1-\va_{m+1})\cdots(1-\va_n)y^{\prime}(1-\va_{m+1})\cdots(1-\va_n)=0$
due to the fact that for each $\ga\in\ovG_n\setminus\ovG_{m,n}$
there exists $k\in\{m+1,\ldots, n\}$ such that $\ga\va_k=\ga$ or
$\va_k\ga=\ga$. Hence, we have
$x=(1-\va_{m+1})\cdots(1-\va_n)y(1-\va_{m+1})\cdots(1-\va_n)$. Now
for any $\be\in G_{n-m}^{\prime}$, $x=\be
x\be^{-1}=(1-\va_{m+1})\cdots(1-\va_n)\be
y\be^{-1}(1-\va_{m+1})\cdots(1-\va_n)$ since
$\be(1-\va_{m+1})\cdots(1-\va_n)=(1-\va_{m+1})\cdots(1-\va_n)\be$.
Then we can write $x$ as
$$
x=\be y\be^{-1}+\sum_{I,J\subseteq\{m+1,\ldots,n\}, I\cup J\neq
\emptyset}(-1)^{|I|+|J|}\va_{I}\be y\be^{-1}\va_{J}.
$$ Clearly the first term belongs to $\mathbb{F}\ovG_{m,n}$ while
the second term is spanned by the elements from
$\ovG_n\setminus\ovG_{m,n}$. Hence $y=\be y\be^{-1}$ for all
$\be\in G_{n-m}^{\prime}$, which implies $y\in\mZ^*_{m,n}$.

Conversely, suppose
$x=(1-\va_{m+1})\cdots(1-\va_n)y(1-\va_{m+1})\cdots(1-\va_n)$ for
some $y\in \mZ^*_{m,n}$. Clearly, we have $x\in I(n,G)$. On the
other hand, since $y\in \mZ^*_{m,n}$, $x$ is invariant under the
conjugation by the elements of $G_{n-m}^{\prime}$. Observe that
$x$ is both left and right annihilated by
$\va_{m+1},\ldots,\va_n$. This implies that $x$ commutes with
$\ovG_{n-m}^{\prime}$, and hence $x\in \overline{\mZ}_{m,n}$.
\end{proof}
 Recall from~(\ref{defn:barJ-m}) that for $\ga\in\ovG_n$ and $1\leq m\leq n$,
 $\bar{J}_m(\ga)=\{i|~m+1\leq i\leq n, \ga_{ii}=1\}$.
\begin{lem}\label{lem:bijection2}
The mapping
$$
\be\mapsto \be\va_{\bar{J}_m(\be)}
$$
defines a bijection of $\ovG_{m,n}$ onto the set of all
$\ga\in\ovG_n$ satisfying the conditions:
\begin{align}
\{j|~m+1\leq j\leq n, \exists~ i\in\mathbb{N}_n \text{ with
}\ga_{ij}\neq 0\}&=\{k|~m+1\leq k\leq n, \exists~ i\in\mathbb{N}_n
\text{ with }\ga_{ki}\neq 0\}\label{cond:rowcolumnnonzero},\\
\text{deg}_m(\ga)&=n-m\label{cond:m-degree}.
\end{align}
\end{lem}
\begin{proof}
Clearly for each $\be\in\ovG_{m,n}$, $\be\va_{\bar{J}_m(\be)}$
satisfy~(\ref{cond:rowcolumnnonzero}) and (\ref{cond:m-degree}).
Conversely, suppose $\ga\in\ovG_n$
satisfy~(\ref{cond:rowcolumnnonzero}) and (\ref{cond:m-degree}).
Then for each $m+1\leq i\leq n$, the $i$th row and column of $\ga$
are zero or nonzero at the same time, and moreover all the
diagonal entries $\ga_{ii}$ are not equal to $1$. Define
$\be\in\ovG_{m,n}$ as follows:
\begin{eqnarray}
\be_{ij}=\left\{
 \begin{array}{ll}
 1, &\text{ if }m+1\leq i=j\leq n \text{ and the }i\text{th row of }\ga \text{ is
 zero},\\
\gamma_{ij},&\text{ otherwise }\notag.
 \end{array}
 \right.
\end{eqnarray}
Then it is easy to see that $\be$ is mapped to $\ga$.
\end{proof}
\begin{lem}\label{lem:injective theta 2}
For $0<m<n$, the restriction of the projection $\theta_n:
\overline{\mZ}_{m,n}\rightarrow \overline{\mZ}_{m,n-1}$ to the
subspace $\Ga_m^{n-m-1}(\overline{\mZ}_{m,n})$ is injective.
\end{lem}
\begin{proof}
Suppose $x\in \overline{\mZ}_{m,n}$ and $\theta_n(x)=0$. We shall
show that $\text{deg}_m(x)=n-m$ unless $x=0$. Since
$x\in\overline{\mZ}_{m,n}$, we have $x\va_n=\va_nx$. Moreover
$\theta_n(x)=0$ implies $\va_nx\va_n=0$, and hence
$x=x-x\va_n=x(1-\va_n)\in I(n,G)\cap \overline{\mZ}_{m,n}$. By
Lemma~\ref{lem:Im(n,G)}, we can write $x$  as a linear combination
of the elements of the type $
(1-\va_{m+1})\cdots(1-\va_n)\be(1-\va_{m+1})\cdots(1-\va_n)$ for
$\be\in\ovG_{m,n}$. Note that
\begin{align}
(1-&\va_{m+1})\cdots(1-\va_n)\be(1-\va_{m+1})\cdots(1-\va_n)\notag\\
&=\sum_{R\cup S\supseteq
\bar{J}_m(\be)}(-1)^{|R|+|S|}\va_R\be\va_S+ \sum_{R\cup
S\nsupseteq
\bar{J}_m(\be)}(-1)^{|R|+|S|}\va_R\be\va_S,\label{division1}
\end{align}
where the summation is over the subsets $R, S$ of
$\{m+1,\ldots,n\}$. Moreover any term in the first sum
of~(\ref{division1}) is of $m$-degree $n-m$ while those in the
second one are of $m$-degree less than $n-m$. Hence it suffices to
show that the elements
$$
\sum_{R\cup S\supseteq \bar{J}_m(\be)}(-1)^{|R|+|S|}\va_R\be\va_S,
\be\in\ovG_{m,n}
$$ are linearly independent. Note that if $R\cup S=
\bar{J}_m(\be)$, then $\va_R\be\va_S=\be\va_{\bar{J}_m(\be)}$ and
$\text{rank}(\be\va_{\bar{J}_m(\be)})=\text{rank}(\be)-|\bar{J}_m(\be)|$;
otherwise if $R\cup S \supset \bar{J}_m(\be)$, then
$\text{rank}(\va_R\be\va_S)<\text{rank}(\be)-|\bar{J}_m(\be)|$.
Hence by the similar argument to the proof of
Lemma~\ref{lem:injective theta}, it suffices to show that for any
fixed $k$ the elements
$$
\sum_{R\cup
S=\bar{J}_m(\be)}(-1)^{|R|+|S|}\va_R\be\va_S=\sum_{R\cup
S=\bar{J}_m(\be)}(-1)^{|R|+|S|}\be\va_{\bar{J}_m(\be)},
 $$
where $ \be\in\overline{G}_{m,n}$ and
$\text{rank}(\be)-|\bar{J}_m(\be)|=k$, are linearly independent.
Lemma~\ref{lem:bijection2} implies that the elements
$\be\va_{\bar{J}_m(\be)}$ are pairwise distinct. Hence it remains
to prove that all the coefficients $\sum_{R\cup
S=\bar{J}_m(\be)}(-1)^{|R|+|S|}$ are nonvanishing, which we refer
to the proof of \cite[Lemma 5.13]{ MO}.
\end{proof}
\begin{prop}\label{prop:basis for ovmZmn}
The elements $\Delta_n^{\Omega,\rho}$ with $\Omega\in\Ga(m,
G\times\mathbb{Z}_+), \rho\in\mathcal{P}(G_*)$ and
$\text{ord}(\Omega)+\|\rho\|\leq n-m$ form a basis of
$\overline{\mZ}_{m,n}$. Moreover, for any $k$ with $0\leq k\leq
n-m$ the elements $\Delta_n^{\Omega,\rho}$ satisfying
$\text{ord}(\Omega)+\|\rho\|\leq k$ form a basis of
$\Ga_m^k(\overline{\mZ}_{m,n})$.
\end{prop}
\begin{proof}
It suffices to prove the second claim and for each fixed $m$, we
shall use the induction on $n$. Note that in the case $n=m$,
$\text{ord}(\Omega)+\|\rho\|\leq n-m$ implies $\Omega\in
\overline{G}_m$ and $\rho=\emptyset$, and hence
$\Delta_n^{\Omega,\rho}=\Omega$ by Remark~\ref{rem:special
elements}. So the proposition holds in the case $n=m$. 
Using Lemma~\ref{lem:injective theta 2} and the argument of
proof of Proposition~\ref{prop:basis for ovmZ0}, we obtain that
for each $0\leq k\leq n-m-1$, the elements
$\Delta_n^{\Omega,\rho}$ with $\text{ord}(\Omega)+\|\rho\|\leq k$
form a basis of $\Ga_m^k(\overline{\mZ}_{m,n})$. Furthermore
following the argument of proof of Proposition~\ref{prop:basis for
ovmZ0}, we can show that
$$
\overline{\mZ}_{m,n}=\Ga_m^{n-m-1}(\overline{\mZ}_{m,n})\oplus
(I(n,G)\cap \overline{\mZ}_{m,n}).
$$
Then it remains to show that the elements $\Delta_n^{\Omega,\rho}$
with $\text{ord}(\Omega)+\|\rho\|= n-m$ form a basis of
$I(n,G)\cap \overline{\mZ}_{m,n}$. By Lemma~\ref{lem:Im(n,G)}, it
is enough to show that these elements are linearly independent.
Since $\text{ord}~\Omega+\|\rho\|= n-m$, we have
$$
\Delta_n^{\Omega,\rho}=(1-\va_{m+1})\cdots(1-\va_n)C_n^{\Omega,\rho}(1-\va_{m+1})\cdots(1-\va_n).
$$
By applying the argument of proof of Proposition~\ref{prop:basis
for ovmZ0}, we can show  that $\Delta_n^{\Omega,\rho}$ are also
linearly independent.
\end{proof}
Recall that for each  $\Omega\in\Ga(m, G\times\mathbb{Z}_+),
\rho\in\mathcal{P}(G_*)$, there exists $\Delta^{\Omega,\rho}\in
\overline{\mZ}_m$ as
$$
\Delta^{\Omega,\rho}=(\Delta_n^{\Omega,\rho}|n\geq m).
$$
By Proposition~\ref{prop:basis for ovmZmn} we have reached the
following.
\begin{thm}\label{thm:basis of ovmZm}
The elements $\Delta^{\Omega,\rho}$  with $\Omega\in\Ga(m,
G\times\mathbb{Z}_+), \rho\in\mathcal{P}(G_*)$ form a basis of the
algebra $\overline{\mZ}_m$. Moreover, for any $k\geq 0$, the
elements $\Delta^{\Omega,\rho}$ with
$\text{ord}(\Omega)+\|\rho\|\leq k$ form a basis of
$\Ga_m^k(\overline{\mZ}_m)$.
\end{thm}
\subsection{The algebraic structure of $\overline{\mZ}_m$}
Recall that for each pair $1\leq k<l\leq n$, $ t_{kl}=\sum_{g\in
G}g^{(k)}(g^{-1})^{(l)}\in\mathbb{F}G^n\subseteq \F\ovG_n$. For
each $k=1,\ldots,m$, we set
\begin{equation}\label{eqn:JM}
u_{k|n}=\sum_{i=k+1}^n t_{ki}(ki)(1-\va_k)(1-\va_i)=\sum_{i=k+1}^n
(1-\va_i)t_{ki}(ki)(1-\va_i).
\end{equation}
Note that the images of $u_{k|n}(1\leq k\leq m)$ under
homomorphism $\Phi$ defined in~(\ref{map:retraction}) are
$\xi_k(1\leq k\leq m)$ the Jucys-Murphy elements for $G_n$; see
(\ref{eqn:jucys}). It is easy to check that $u_{k|n}$ commutes
with $\ovG_{n-m}^{\prime}$ for any $n\geq m$, and hence
$u_{k|n}\in \overline{\mZ}_{m,n}$. Furthermore since
$\theta_n(u_{k|n})=u_{k|n-1}$ and $\text{deg}(u_{k|n})=2$,  we can
define the elements $u_k\in \overline{\mZ}_m$ for $1\leq k\leq m$
as the sequence
$$
u_k=(u_{k|n}|n\geq m).
$$

Recall that the algebra  $\F\ovG_m$ is generated by $g\in G^m,
s_1,\ldots,s_{m-1}$ and $\va_1,\ldots,\va_m$. Lemma~\ref{lem:ovmZm
to FovGm} says that the algebra  $\F\ovG_m$ can be naturally
embedded in $\overline{\mZ}_m$.  Then we have the following lemma.
\begin{lem}\label{lem:relations of Htilde}
The following relations hold in the algebra $\overline{\mZ}_m$:
\begin{align}
gu_k&=u_k,\quad \forall g\in G^m, 1\leq k\leq m\label{eqn:gandu}\\
s_ku_k&=u_{k+1}s_k+t_{k,k+1}(1-\va_k)(1-\va_{k+1}),
s_ku_l=u_ls_k, l\neq k,k+1, 1\leq k\leq m-1\label{eqn:skand u}\\
u_ku_l&=u_lu_k,\quad \va_ku_k=u_k\va_k=0,\quad,
\va_lu_k=u_k\va_l,\quad 1\leq l\neq k\leq m,\label{eqn:uandu}
\end{align}
\end{lem}
\begin{proof}
Since $u_k=(u_{k|n}|n\geq m)$ for each $1\leq k\leq m$, it
suffices to check these relations with $u_k$ replaced by $u_{k|n}$
for each $1\leq k\leq m$ in the algebra $\overline{\mZ}_{m,n}$ for
any $n\geq m$. Note that there is a canonical embedding of
$\F\ovG_l$ in $\F\ovG_n$ with image $\mathbb{F}\ovG_l^{\prime}$
for each $1\leq l\leq n$ and let us denote by
$\Delta_l^{\prime\rho}$ the corresponding image of the element
$\Delta_l^{\rho}$ for each $\rho\in\mathcal{P}(G_*)$. Then we have
$u_{k|n}=\Delta_{n-k+1}^{\prime((2),\emptyset,\ldots,\emptyset)}-
\Delta_{n-k}^{\prime((2),\emptyset,\ldots,\emptyset)}$ for each
$1\leq k\leq m$, and hence $u_{1|n},\ldots,u_{m|n}$ pairwise
commute. The remaining relations can be easily checked by applying
Proposition~\ref{prop:pre. of ovGn}.
\end{proof}

Let us denote by $\overline{\mathcal{H}}_m(G)$ the subalgebra of
$\overline{\mZ}_m$ generated by $\F\ovG_m$ and the elements
$u_1,\ldots,u_m$. The following theorem describes the structure of
the algebra $\overline{\mZ}_m$.
\begin{thm}\label{thm:tensor1} We have an algebra isomorphism
$$
\overline{\mZ}_m\cong
\overline{\mZ}_0\otimes\overline{\mathcal{H}}_m(G).
$$
Moreover, $\overline{\mathcal{H}}_m(G)$ is isomorphic to the
algebra with generators $g\in G^m, s_1,\ldots,s_{m-1}$ and
$\va_1,\ldots,\va_m$ subject to the
relations~(\ref{reln:braid})-(\ref{reln:gp.va}) and
(\ref{eqn:gandu})-(\ref{eqn:uandu}).
\end{thm}
\begin{proof}
Recall from~(\ref{defn:elementC}) the definition of
$\Delta_n^{\Omega,\rho}$ for any
$\Omega\in\Ga(m,G\times\mathbb{Z}_+)$ and
$\rho\in\mathcal{P}(G_*)$ with $\text{ord}(\Omega)+\|\rho\|\leq
n-m$. Note that the $m$-degree of $\Delta_n^{\Omega,\rho}$ is
$\text{ord}(\Omega)+\|\rho\|$. It is easy to see that
$$
\Delta_n^{\Omega,\emptyset}\Delta_n^{\mathbb{I},\rho}=\Delta_n^{\Omega,\rho}+(\cdots),
$$
where $(\cdots)$ is a linear combination of elements
$\Delta_n^{\Omega^{\prime},\rho^{\prime}}$ with
$\text{ord}(\Omega^{\prime})+\|\rho^{\prime}\|<\text{ord}(\Omega)+\|\rho\|$.
Furthermore, by Remark~\ref{rem:special elements}(2), we have:
\begin{equation}\label{difference 1}
\Delta_n^{\rho}-\Delta_n^{\mathbb{I},\rho} =\sum_{
|T|=\|\rho\|}C^{\rho}_{n,T}\overline{\va}_T,
\end{equation}
where the summation is over the subset $T$ of $\mathbb{N}_n$ not
contained in $\{m+1,\ldots,n\}$. Since each element appearing on
the right hand side of~(\ref{difference 1}) is of $m$-degree less
than $\|\rho\|$, we obtain that
$$
\Delta_n^{\Omega,\emptyset}\Delta_n^{\rho}=\Delta_n^{\Omega,\rho}+(\cdots),
$$
where $(\cdots)$ is a linear combination of elements
$\Delta_n^{\Omega^{\prime},\rho^{\prime}}$ with
$\text{ord}(\Omega^{\prime})+\|\rho^{\prime}\|<\text{ord}(\Omega)+\|\rho\|$.
Then Proposition~\ref{prop:basis for ovmZmn} implies that the
elements $\Delta_n^{\Omega,\emptyset}\Delta_n^{\rho}$ with
$\Omega\in\Ga(m,G\times\mathbb{Z}_+)$ and
$\rho\in\mathcal{P}(G_*)$ such that
$\text{ord}(\Omega)+\|\rho\|\leq n-m$ form a basis of
$\overline{\mZ}_{m,n}$. Hence
$\Delta^{\Omega,\emptyset}\Delta^{\rho}$ with
$\Omega\in\Ga(m,G\times\mathbb{Z}_+)$ and
$\rho\in\mathcal{P}(G_*)$ form a basis of the algebra
$\overline{\mZ}_m$. This means $\overline{\mZ}_m$ is a free
$\overline{\mZ}_0$-module with the basis
$\{\Delta^{\Omega,\emptyset}|~\Omega\in\Ga(m,G\times\mathbb{Z}_+)\}$
by Theorem~\ref{thm:basis for vmZ0}.

If $\Omega\in\Ga(m,G\times\mathbb{Z}_+)$ has zero rows
$i_1,\ldots,i_s,$ then we can properly replace some $0$'s in these
rows by $1$ to get an element $\Omega^{\prime}\in
S(m,G\times\mathbb{Z}_+)$ such that
$\Omega=\va_{i_1}\cdots\va_{i_s}\Omega^{\prime}$. Then
$\text{ord}(\Omega)=\text{ord}(\Omega^{\prime})$ and for each
subset $P\subseteq\{m+1,\ldots,n\}$ with $|P|=\text{ord}(\Omega)$
we have $\Ga(\Omega,
P)=\{\va_{i_1}\cdots\va_{i_s}\ga|~\ga\in\Ga(\Omega^{\prime},P)\}$.
Hence
$$
\Delta_n^{\Omega,\emptyset}=\va_{i_1}\cdots\va_{i_s}\Delta_n^{\Omega^{\prime},\emptyset}.
$$
On the other hand, each $\Omega^{\prime}\in
S(m,G\times\mathbb{Z}_+)$ can be written as:
$$
\Omega^{\prime}=\be\alpha_1^{k_1}\cdots\alpha_m^{k_m},
$$
where $\be\in G_m$ and $\alpha_i\in \Ga(m,G\times\mathbb{Z}_+)$ is
the diagonal matrix whose $(i,i)$th entry is $z$ and all other
diagonal entries are equal to $1$ and $k_i\in\mathbb{Z}_+$ for
each $1\leq i\leq m$. Then we have
$$
\Delta_n^{\Omega^{\prime},\emptyset}=\Delta_n^{\be,\emptyset}
(\Delta_n^{\alpha_1,\emptyset})^{k_1}\cdots(\Delta_n^{\alpha_m,\emptyset})^{k_m}+(\cdots),
$$
where $(\cdots)$ is a linear combination of elements
$\Delta_n^{\Omega^{''},\rho^{''}}$ with
$\text{ord}(\Omega^{''})+\|\rho^{''}\|<\text{ord}(\Omega^{'})$. By
Remark~\ref{rem:special elements}(1), for $\be\in G_m$ we have
$$
\Delta_n^{\be,\emptyset}=\be.
$$
For each $1\leq k\leq m$, $\text{ord}(\alpha_k)=1$, then it can be
checked that
\begin{equation}\label{eq:Deltak}
\Delta_n^{\alpha_k,\emptyset}=\sum_{l=m+1}^n(1-\va_l)t_{kl}(kl)(1-\va_l).
\end{equation}
Indeed, in the case of $\alpha_1$ and $P=\{l\}$ with $m+1\leq
l\leq n$, we have
$$\Ga(\alpha_1,P)=\{\ga\in\ovG_{m,n}|\ga_{1l}\ga_{l1}=1,\ga_{ij}=\delta_{ij},
i,j\neq 1,l\}\subseteq G_n,
$$ and it is easy to see
$$
\sum_{\ga\in\Ga(\alpha_1,P)}\overline{\va}_{P}\ga\overline{\va}_P=(1-\va_{l})t_{1l}(1l)(1-\va_{l}).
$$
Hence
$$
\Delta_n^{\alpha_1,\emptyset}=\sum_{l=m+1}^n(1-\va_{l})t_{1l}(1l)(1-\va_{l}).
$$
Then by~(\ref{eq:Deltak}) and definition of $u_{k|n}$
in~(\ref{eqn:JM}) we have for each $1\leq k\leq m$
$$
\Delta_n^{\alpha_k,\emptyset}= u_{k|n}+\text{elements of
}m\text{-degree zero}.
$$
By now we can conclude that $\be u_{1|n}^{k_1}\cdots
u_{m|n}^{k_m}$ coincides with $\Delta_n^{\be,\emptyset}
(\Delta_n^{\alpha_1,\emptyset})^{k_1}\cdots(\Delta_n^{\alpha_m,\emptyset})^{k_m}$
modulo lower $m$-degree terms and hence $\overline{\mZ}_m$ as
$\overline{\mZ}_0$-module has a basis
$$
\{\ga u_1^{k_1}\cdots u_m^{k_m}|~\ga\in \ovG_m,
k_l\in\mathbb{Z}_+, 1\leq l\leq m\}.
$$
Note that the subalgebra $\overline{\mathcal{H}}_m(G)$ is spanned
by $\ga u_1^{k_1}\cdots u_m^{k_m}$ with $\ga\in \ovG_m,
k_l\in\mathbb{Z}_+, 1\leq l\leq m$. Hence $\overline{\mZ}_m\cong
\overline{\mZ}_0\otimes \overline{\mathcal{H}}_m(G)$. Moreover the
subset
$$
\{\ga u_1^{k_1}\cdots u_m^{k_m}|~\ga\in \ovG_m,
k_l\in\mathbb{Z}_+, 1\leq l\leq m\}
$$
form a basis of the subalgebra $\overline{\mathcal{H}}_m(G)$ and
hence the second claim follows since there exists a canonical
epimorphism from the abstract algebra with generators $g\in G^m,
s_1,\ldots,s_{m-1}$ and $\va_1,\ldots,\va_m$ subject to the
relations~(\ref{reln:braid})-(\ref{reln:gp.va}) and
(\ref{eqn:gandu})-(\ref{eqn:uandu}) to the algebra
$\overline{\mathcal{H}}_m(G)$ by Lemma~\ref{lem:relations of
Htilde}.
\end{proof}
\section{The centralizers $\mZ_{m,n}$ and wreath Hecke
algebras}\label{CwRh}

In this section we shall formulate the connection between the
centralizer $\mZ_{m,n}$ of $G_{n-m}^{\prime}$ in $\F G_n$ and the
notion of wreath Hecke algebras $\mathcal{H}_m(G)$ introduced
in~\cite{WW}.
%
%
%
\subsection{The connection between $\mZ_{m,n}$ and $\mathcal{H}_m(G)$}
%

The following lemma will be used later on.
\begin{lem}\label{map:ovH to mZ}
The mapping
\begin{align}
\overline{\Psi}:\overline{\mathcal{H}}_m(G)&\longrightarrow \overline{\mZ}_{m,n}\\
g\mapsto g, s_l\mapsto s_l, \va_k&\mapsto\va_k,u_k\mapsto u_{k|n},
1\leq l\leq m-1, 1\leq k\leq m
\end{align}
defines an algebra homomorphism. Moreover, the algebra
$\overline{\mZ}_{m,n}$ is generated by the image of $\varphi$ and
the center of the subalgebra $\mathbb{F}\ovG_{n-m}^{\prime}$.
\end{lem}
\begin{proof}
The first claim is obvious. It follows from the proof of
Theorem~\ref{thm:tensor1} that $\overline{\mZ}_{m,n}$ is generated
by
$\ga(u_{1|n})^{k_1}\cdots(u_{m|n})^{k_m}\Delta_n^{\mathbb{I},\rho}$
with $\ga\in \ovG_m, k_1,\ldots,k_m\in\mathbb{Z}_+$ and
$k_1+\cdots+k_m+\|\rho\|\leq n-m$. This implies that 
$\overline{\mZ}_{m,n}$ is generated by the image of  $\varphi$ and
the elements $\Delta_n^{\mathbb{I},\rho}$ with $\|\rho\|\leq n-m$.
By Remark~\ref{rem:special elements}, we have
$\Delta_n^{\mathbb{I},\rho}=\sum_{T\subseteq\{m+1,\ldots,n\},
|T|=\|\rho\|}C_{n,T}^{\rho}\overline{\va}_T.$  Observe that the
algebra $\F\ovG_{n-m}^{\prime}$ is isomorphic to $\F\ovG_{n-m}$.
Hence the elements $\Delta_n^{\mathbb{I},\rho}$ with
$\rho\in\mathcal{P}(G_*)$ and $\|\rho\|\leq n-m$ form a basis of
center of the subalgebra $\F\ovG_{n-m}^{\prime}$ by
Proposition~\ref{prop:basis for ovmZ0}.
\end{proof}
The wreath Hecke algebra $\mathcal{H}_m(G)$ (see \cite{WW}) is
defined to be the algebra with generators $g\in G^m,
s_1,\ldots,s_{m-1}$ and $x_1, \ldots, x_m$ with the defining
relations~(\ref{reln:braid})-(\ref{reln:gp.va}) and
\begin{align*}
gx_k&=x_kg,\\
s_kx_k&=x_{k+1}s_k+t_{k,k+1},\quad s_kx_l=x_ls_k, \quad l\neq k,k+1\\
x_kx_l&=x_lx_k.
\end{align*}
Then the following is obvious from Theorem~\ref{thm:tensor1}.
\begin{lem}\label{lem:ovH to H}
The mapping
$$
g\mapsto g, \quad s_k\mapsto s_k, \quad u_l\mapsto x_l,\quad
\va_l\mapsto 0
$$
defines an algebra epimorphism
$\overline{\Phi}:\overline{\mathcal{H}}_m(G)\rightarrow
\mathcal{H}_m(G)$.
\end{lem}


%


Recall from Lemma~\ref{lem:ovmZmn to Zmn} that the retraction
homomorphism $\Phi$ in~(\ref{map:retraction}) maps
$\overline{\mZ}_{m,n}$ onto $\mZ_{m,n}$. Then the following lemma
 can be easily checked by Lemma~\ref{map:ovH to mZ} and Lemma~\ref{lem:ovH to H}.

\begin{prop}\label{cor:Bm(n,G)}
The mapping
\begin{align*}
\Psi: \mathcal{H}_m(G)&\longrightarrow \mZ_{m,n},\\
 g\mapsto g,
s_k&\mapsto (k,k+1), x_l\mapsto\sum_{i=l+1}^nt_{li}
\end{align*}
is an algebra homomorphism. Moreover the following diagram is
commutative:
$$\unitlength=1cm
\begin{picture}(6,2.5)
\put(1.1,2){$\overline{\mathcal{H}}_m(G)$}
\put(3.9,2){$\overline{\mZ}_{m,n}$}
\put(1.1,0.2){$\mathcal{H}_m(G)$} \put(3.9,0.2){$\mZ_{m,n}$}
\put(2.3,2.1){\vector(1,0){1.5}} \put(2.3,0.3){\vector(1,0){1.5}}
\put(2.70,2.2){$\overline{\Psi}$} \put(2.70,0.4){$\Psi$}
\put(1.5,1.85){\vector(0,-1){1.3}}
\put(4.4,1.85){\vector(0,-1){1.3}}
\put(1.03,1.1){$\overline{\Phi}$} \put(4.6,1.1){$\Phi$}
\end{picture}$$
Hence the algebra $\mZ_{m,n}$ is generated by the center of the
subalgebra $\mathbb{F}G_{n-m}^{\prime}$ and the image of $\Psi$.
\end{prop}




%
%
\subsection{The Gelfand-Zetlin algebra for $G_n$.}\label{GZ algebra}
%
%
%
Recall that  $\mZ_{0,m}$ is the center of the algebra
$\mathbb{F}G_m$. Let us denote by $\mZ^{\prime}_{0,m}$ the center
of the subalgebra $\mathbb{F}G^{\prime}_m$ for each $0\leq m\leq
n$. Clearly we have $\mZ_{0,m}\cong\mZ^{\prime}_{0,m}$.
\begin{lem}\label{lem:G times Gn-1}
The centralizer $(\mathbb{F}G_n)^{G\times G_{n-1}^{\prime}}$ of
$G\times G_{n-1}^{\prime}$ in $\mathbb{F}G_n$ is generated by
$\mZ^{\prime}_{0,n-1}$, $\mZ_{0,1}$ and $\xi_1$. In particular,
$(\mathbb{F}G_n)^{G\times G_{n-1}^{\prime}}$ is commutative.
\end{lem}
\begin{proof}
It is easy to check that $\mZ^{\prime}_{0,n-1}$, $\mZ_{0,1}$ and
$\xi_1$ belong to $(\mathbb{F}G_n)^{G\times G_{n-1}^{\prime}}$. On
the other hand, $(\mathbb{F}G_n)^{G\times G_{n-1}^{\prime}}$ is
isomorphic to the centralizer of $G=G_1$ in the subalgebra
$(\mathbb{F}G_n)^{G^{\prime}_{n-1}}$ which is generated by $
\mZ^{\prime}_{0,n-1}, G_1$ and $\xi_1$ by
Proposition~\ref{cor:Bm(n,G)}. Observe that $G$ commutes with
$\mZ^{\prime}_{0,n-1}$ and $\xi_1$, and hence
$(\mathbb{F}G_n)^{G\times G_{n-1}^{\prime}}$ is contained in the
subalgebra of $(\mathbb{F}G_n)^{G^{\prime}_{n-1}}$ generated by
$\mZ^{\prime}_{0,n-1}$, $\xi_1$ and $\mZ_{0,1}$.
\end{proof}
From now on we assume that
$\F=\mathbb{C}$. The following lemma is standard.
\begin{lem}\cite[Lemma 1.0.1]{K}\label{lem:from K}
Let $\mathcal{B}\subseteq \mathcal{A}$ be semisimple finite
dimensional $\mathbb{F}$-algebras and let $\mathcal{C}$ be the
centralizer of $\mathcal{B}$ in $\mathcal{A}$. If $V$ is
irreducible over $\mathcal{A}$ and $W$ is irreducible over
$\mathcal{B}$, then
$$
\text{Hom}_{\mathcal{B}}(W,\text{res}^{\mathcal{A}}_{\mathcal{B}}V)
$$
is irreducible over $\mathcal{C}$.
\end{lem}

%
\begin{cor}\label{cor:mult.free} Let $V$ be an irreducible $\mathbb{C}G_n$-module. Then
the restriction $\text{res}^{\mathbb{C} G_n}_{\mathbb{C} (G\times
G_{n-1}^{\prime})}V$ is multiplicity free.
\end{cor}
\begin{proof}
By Lemma~\ref{lem:from K}, $\text{Hom}_{\mathbb{C} (G\times
G_{n-1}^{\prime})}(W,\text{res}^{\F G_n}_{\mathbb{C} (G\times
G_{n-1}^{\prime})}V)$ is irreducible over the centralizer
$(\mathbb{C}G_n)^{G\times G_{n-1}^{\prime}}$ for any irreducible
$\mathbb{C} (G\times G_{n-1}^{\prime})$-module $W$ and hence has
dimension one or zero since the centralizer
$(\mathbb{C}G_n)^{G\times G_{n-1}^{\prime}}$ is commutative by
Lemma~\ref{lem:G times Gn-1}. Then it follows that the restriction
$\text{res}^{\mathbb{C} G_n}_{\mathbb{C} G\times
G_{n-1}^{\prime}}V$ is multiplicity free.
\end{proof}

Suppose $\{V_1, V_2,\ldots,V_r\}$ is a complete set of
nonisomorphic irreducible $\mathbb{C}G$-modules and fix a basis
$\Pi_i$ for $V_i$ with $1\leq i\leq r$.


 Now decompose the algebra $\mathbb{C}G_n$ according to
the Wedderburn theorem
$$
\mathbb{C}G_n=\bigoplus_{U}\mbox{End}_{\mathbb{C}}(U),
$$
where the sum is over the representation of the isoclasses of
irreducible $\mathbb{C}G_n$-modules. Let $U$ be an irreducible
$\mathbb{C}G_n$-module. Then Corollary~\ref{cor:mult.free} implies
that $\text{res}^{\mathbb{C} G_n}_{\mathbb{C} G\times
G^{\prime}_{n-1}}U=\bigoplus_{1\leq k\leq t}V_{i_k}\bigotimes
W_{i_k}$ for some $t\in\mathbb{Z}_+$ and irreducible
$\mathbb{C}G_{n-1}$-modules $W_{i_1},\ldots,W_{i_t}$. Furthermore,
$V_{i_k}\otimes W_{i_k} \ncong V_{i_s}\otimes W_{i_s} $ for $1\leq
k\neq s\leq t$ as $G\times G^{\prime}_{n-1}$-modules.

Decomposing each $W_{i_k}$ on restriction to $G\times
G^{\prime}_{n-2}$, and continuing inductively all the way to
$G\times G\times\ldots\times G$, we get a canonical decomposition
$$
\text{res}_{G^n}U=\bigoplus_{T} V_{T_1}\otimes
V_{T_2}\otimes\ldots\otimes V_{T_n}$$ where $T_j\in\{1,2,\ldots,r\}$ for
$1\leq j\leq n$, into irreducible $\mathbb{C}G^n$-modules, where
$T$ runs over all possibilities:
$$T: V_{T_1}\otimes\ldots\otimes
V_{T_n}\rightarrow  V_{T_1}\otimes \ldots\otimes
V_{T_{n-2}}\otimes W_n\rightarrow\cdots\rightarrow V_{T_1}\otimes
W_2\rightarrow U.$$ Each possible $T$ is called a path. We already
have a fixed basis $\Pi_i$ for each $V_i$. So we have a fixed
basis $\Pi_{T}$ for $ V_{T_1}\otimes V_{T_2}\otimes\ldots\otimes V_{T_n}$
for each possible path $T$. Finally, we get a basis $\Pi_U$ for
each irreducible $\mathbb{C}G_n$-module $U$, which is called a
{\em Gelfand-Zetlin} basis for $U$. Then we can identify
$$\mathbb{C}G_n=\bigoplus_{U}\text{M}_{\text{dim}U}(\mathbb{C}).$$

Define the {\em Gelfand-Zetlin} algebra
$\mathcal{A}_n\subseteq\mathbb{C}G_n$ as the subalgebra consisting
of all elements of $\mathbb{C}G_n$, which are diagonal with
respect to our fixed basis in every irreducible
$\mathbb{C}G_n$-module. Clearly, $\mathcal{A}_n$ is a maximal
commutative subalgebra of $\mathbb{C}G_n$. Also, $\mathcal{A}_n$
is semi-simple.  Note that in the case when $G=\{1\}$,
$\mathcal{A}_n$ coincides with the Gelfand-Zetlin algebra for the
symmetric group $S_n$.

On the other hand by taking the fixed basis $\Pi_i$ for $V_i$ with
$1\leq i\leq r$ we can also decompose the algebra $\mathbb{C}G$
according to the Wedderburn Theorem as
$$\mathbb{C}G=\bigoplus_{1\leq i\leq
r}\mbox{M}_{\mbox{dim}V_i}(\mathbb{C}).$$ Define $\La$ to be the
subalgebra which consists of all elements of $\mathbb{C}G$, which
are diagonal with respect to our fixed basis $\Pi_i$ in each $V_i$
for $1\leq i\leq r$. That means $\La$ consists of all diagonal
matrices under the above identification. Clearly, $\La$ is a
maximal commutative subalgebra of $\mathbb{C}G$. Also, $\La$ is
semi-simple. Denote by $\Lambda^{(k)}$ the corresponding
subalgebra of $\mathbb{C} G^{(k)}$ in $\mathbb{C} G_n$ for $1\leq
k\leq n$.

\begin{thm}\label{thm:GZ algebra} Retain the above notations. Then we have

\begin{enumerate}

\item $\mathcal{A}_n$ is generated by the subalgebras
$\mZ^{\prime}_{0,1},\ldots,\mZ^{\prime}_{0,n},
\La^{(1)},\ldots,\La^{(n)}$.

\item $\mathcal{A}_n$ is generated by $\La^{(1)},\ldots,\La^{(n)}$
and the JM-elements $\xi_1,\ldots,\xi_n$.
\end{enumerate}
\end{thm}
\begin{proof}
Let $e_U$ be the central idempotent of $\mathbb{C}G_n$, which acts
as identity on $U$ and as zero on any irreducible
$\mathbb{C}G_n$-modules $U^\prime \ncong U$. If $T:
V_{T_1}\otimes\ldots\otimes V_{T_n}\rightarrow  V_{T_1}\otimes
\ldots\otimes V_{T_{n-2}}\otimes U_n\rightarrow\cdots\rightarrow
V_{T_1}\otimes U_2\rightarrow U$ is one possibility of
decomposition, then $e_{V_{T_1}}^{(1)}e_{V_{T_2}}^{(2)}\cdots
e_{V_{T_n}}^{(n)}\centerdot e_{V_{T_1}}^{(1)}\cdots
e_{V_{T_{n-2}}}^{(n-2)}e_{U_n}\centerdot\cdots\cdot
e_{V_{T_1}}^{(1)}e_{U_2}\centerdot e_U$ acts on $U$ the projection to
$V_{T_1}\otimes V_{T_2}\otimes\ldots\otimes V_{T_n}$ and as zero
on any irreducible $\mathbb{C}G_n$-module  $U^\prime \ncong U$.

Let $v_{j_1}\otimes v_{j_2}\cdots v_{j_n}\in\Pi_{T}$ and let
$f_{j_k}\in \La\subseteq\mathbb{C}G$ be the element acting on
$V_{T_k}$ as the projection to $v_{j_k}$ and acting as zero on any
irreducible $\mathbb{C}G$-module $V^\prime\ncong V_{T_k}$ for each
$1\leq k\leq n$. Then $f_{j_1}^{(1)}f_{j_2}^{(2)}\cdots
f_{j_n}^{(n)}$ acts as the projection to $v_{j_1}\otimes
v_{j_2}\cdots v_{j_n}$ of $\mathbb{F}G^n$-module $V_{T_1}\otimes
V_{T_2}\otimes\ldots\otimes V_{T_n}$. Therefore
$f_{j_1}^{(1)}f_{j_2}^{(2)}\cdots f_{j_n}^{(n)}\centerdot
e_{V_{T_1}}^{(1)}e_{V_{T_2}}^{(2)}\cdots
e_{V_{T_n}}^{(n)}\centerdot e_{V_{T_1}}^{(1)}\cdots
e_{V_{T_{n-2}}}^{(n-2)}e_{U_n}\centerdot\cdots\cdot
e_{V_{T_1}}^{(1)}e_{U_2}\centerdot e_U$ acts as the projection of
$U$ to the one-dimensional subspace $\mathbb{F}v_{j_1}\otimes
v_{j_2}\otimes\cdots \otimes v_{j_n}$ and acts as zero on other irreducible
$\mathbb{F}G_n$-module  $U^\prime \ncong U$. However
$f_{j_1}^{(1)}f_{j_2}^{(2)}\cdots f_{j_n}^{(n)}\centerdot
e_{V_{T_1}}^{(1)}e_{V_{T_2}}^{(2)}\cdots
e_{V_{T_n}}^{(n)}\centerdot e_{U_2}e_{V_{T_3}}^{(3)}\cdots
e_{V_{T_n}}^{(n)}\centerdot\cdots\centerdot
e_{U_{n-1}}e_{V_{T_n}}^{(n)}\centerdot e_V \in \langle
\mZ^{\prime}_{0,1},\ldots,\mZ^{\prime}_{0,n},
\La^{(1)},\ldots,\La^{(n)}\rangle$ since $e_{U_m}\in
\mathcal{Z}^{\prime}_{0,n-m+1}$. That means that $\mathcal{A}_n$
is contained in the subalgebra generated by $
\mZ^{\prime}_{0,1},\ldots,\mZ^{\prime}_{0,n},$ and $
\La^{(1)},\ldots,\La^{(n)}$. But we already know $\mathcal{A}_n$
is a maximal commutative subalgebra of $\mathbb{C}G_n$. Hence
$\mathcal{A}_n= \langle
\mZ^{\prime}_{0,1},\ldots,\mZ^{\prime}_{0,n},
\La^{(1)},\ldots,\La^{(n)}\rangle$.

Note that $\xi_k$ is the sum of all elements with type
$\rho=(\rho(C))_{C\in G_*}$, where $\rho(1)=(21^{k-2}),
\rho(C)=\emptyset$ for $C\neq 1$ in $\mathbb{C}G^{\prime}_k$ minus
the sum of all elements with type $\tau=(\tau(C))_{C\in G_*}$,
where $\tau(1)=(21^{k-3}), \tau(C)=\emptyset$ for $C\neq 1$ in
$\mathbb{C}G^{\prime}_{k-1}$. So $\xi_k\in\langle
\mZ^{\prime}_{0,k},\mZ^{\prime}_{0,k-1}\rangle$. Thus $\xi_k\in
\mathcal{A}_n$. Therefore $\langle \xi_1,\ldots,\xi_n,
\La^{(1)},\ldots,\La^{(n)}\rangle \subseteq\mathcal{A}_n$.  We
shall use the induction on $n$ to prove $ \mathcal{A}_n\subseteq
\langle \xi_1,\ldots,\xi_n, \La^{(1)},\ldots,\La^{(n)}\rangle$ .
Identify $\mathbb{C} G_{n-1}$ with $\mathbb{C} G^{\prime}_{n-1}$
in $\mathbb{C} G_n$ and let us denote by
$\mathcal{A}^{\prime}_{n-1}$ the subalgebra in $\mathbb{C}
G^{\prime}_{n-1}$ corresponding to $\mathcal{A}_{n-1}$ in
$\mathbb{C}G_{n-1}$.  Then we have $\mathcal{A}^{\prime}_{n-1}=
\langle \xi_2,\ldots,\xi_n, \La^{(2)},\ldots,\La^{(n)}\rangle$. So
it follows from $(1)$ that
$\mathcal{A}_n=\langle\xi_2,\ldots,\xi_n,
\La^{(2)},\ldots,\La^{(n)}, \mZ^{\prime}_{0,n}, \La^{(1)}\rangle$.
It suffices to show $\mZ^{\prime}_{0,n}\subseteq\langle
\xi_1,\xi_2\ldots,\xi_n, \La^{(1)},\ldots,\La^{(n)}\rangle$.

We know $\mZ^{\prime}_{0,n}$ is the center of $\mathbb{C} G_n$,
and hence $\mZ^{\prime}_{0,n}\subseteq \mathbb{C} G_n^{G\times
G^{\prime}_{n-1}}=\langle \mZ^{\prime}_{0,n-1}, \xi_1,
\mathcal{Z}_1\rangle\subseteq \langle
\mathcal{Z}^{\prime}_{0,n-1}, \xi_n, \La^{(1)}\rangle$ since
$\mathcal{Z}_{0,1}\subseteq \La^{(1)}$. But
$\mathcal{Z}^{\prime}_{0,n-1}\subseteq
\mathcal{A}^{\prime}_{n-1}\subseteq \langle \xi_2,\ldots,\xi_{n},
\La^{(2)},\ldots,\La^{(n)}\rangle$. So
$\mathcal{Z}^{\prime}_{0,n-1}\subseteq\langle \xi_1,\ldots,\xi_n,
\La^{(1)},\ldots,\La^{(n)}\rangle$ as desired.
\end{proof}
\begin{rem} Note that Theorem~\ref{thm:GZ algebra} specializes to
\cite[Theorem 11]{O1} when $G=\{1\}$. Then we can extend the new
approach for representations of symmetric groups established by
Okounkov and Vershik in~\cite{OV} to the wreath products by taking
the advantage of the classification of finite dimensional
irreducible $\mathcal{H}_m(G)$-modules obtained in~\cite{WW}. We
omit this part in the present paper as it overlaps significantly
with~\cite{Pu}.

\end{rem}

%


\begin{thebibliography}{ABC}

\bibitem[Dr]{Dr} V.G. Drinfeld, {\em Degenerate affine Hecke algebras and
Yangians}, Functional Anal. Appl. {\bf 20} (1986), 62--64.

\bibitem[JK]{JK} G. James and A. Kerber, {\em The Representation Theory of
the Symmetric Groups}, Addison-Wesley, London, 1980.

\bibitem[K]{K} A. Kleshchev, {\em Linear and Projective
Representations of Symmetric Groups}, Cambridge University Press,
2005.

\bibitem[Lus]{Lus} G. Lusztig, {\em Affine Hecke algebras and their
graded version}, J. Amer. Math. Soc. {\bf 2} (1989), 599--635.


\bibitem[Mac]{Mac} I.G. Macdonald, {\em Symmetric functions and Hall
polynomials}, Second edition, Clarendon Press, Oxford, 1995.



\bibitem[MO]{MO} A.I. Molev and G.I. Olshanski, {\em Degenerate Affine Hecke
algebras and centralizer construction for the symmetric groups},
J. Algebra {\bf 237} (2001), 302--341.

\bibitem[O1]{O1} G.I. Olshanski, {\em Extension of the algebra $U(g)$ for infinite-
dimensional classical Lie algebras $g$, and the Yangians
$Y(gl(m))$}, Soviet Math. Dokl. {\bf 36}, No. 3 (1988), 569--573.

\bibitem[O2]{O2} G.I. Olshanski, {\em Representations of inifinite-dimensional classical groups,
limits of enveloping algebras, and Yangians}, in: A.A. Kirillov
(Ed.), Topics in Representation Theory, Advances in Soviet
Mathematics, Vol. 2, AMS, Providence, RI, 1991, 1--66.


\bibitem[OV]{OV} A. Okounkov and A. Vershik, {\em A new approach to
representation theory of symmetric groups}, Selecta Math. (N.S.)
{\bf 2} (1996), 581--605.

\bibitem[Pu]{Pu} I.A. Pushkarev, {\em On the representation theory of wreath
products of finite groups and symmetric groups}, J. Math. Sci.
{\bf 96} (1999), 3590--3599.

\bibitem[RS]{RS} A. Ram and A. Shepler, {\em Classification of graded
Hecke algebras for complex reflection groups}, Comment. Math.
Helv. {\bf 78} (2003), 308--334.

\bibitem[WW]{WW} J. Wan and W. Wang, {\em Modular representations and
branching rules for wreath Hecke algebras}, Int. Math. Res. Not. IMRN 2008,
Art. ID rnn128, 31 pp.

\bibitem[W1]{W1} W. Wang,
{\em Vertex algebras and the class algebras of wreath products},
Proc. London Math. Soc. {\bf 88} (2004), 381--404.

\bibitem[W2]{W2} W. Wang,
{\em The Farahat Higman ring of wreath products and Hilbert
schemes}, Adv. in Math. {\bf 187} (2004), 417-446.

\end{thebibliography}
\end{document}